\documentclass{springer-article}

\usepackage{bm}
\usepackage{tikz}
\usepackage{graphicx}%
\usepackage{multirow}%
\usepackage{amsmath,amssymb,amsfonts}%
\usepackage{amsthm}%
\usepackage{mathrsfs}%
\usepackage[title]{appendix}%
\usepackage{xcolor}%
\usepackage{textcomp}%
\usepackage{manyfoot}%
\usepackage{booktabs}%
\usepackage{algorithm}%
\usepackage{algorithmicx}%
\usepackage{algpseudocode}%
\usepackage{listings}%
\usepackage{comment}
\usepackage{multicol}
\usepackage{extarrows}
\usepackage{multirow}

\usepackage{enumitem} 

\usepackage{tikz-cd}

\DeclareMathOperator*{\argmax}{arg\,max}

 \newcommand{\rom}[1]{
  \textup{\lowercase\expandafter{\romannumeral#1}} }
  
\newcommand{\Rom}[1]{
  \textup{\uppercase\expandafter{\romannumeral#1}} }

\newcommand{\doublestar}{{\star \star}}
\newcommand{\reals}{{\mathbb R}}

\newcommand{\bbr}{\reals}
\newcommand{\bbz}{{\mathbb Z}}
\newcommand{\bbn}{{\mathbb N}}
\newcommand{\PP}{\mathbb{P}}
\newcommand{\EE}{\mathbb{E}}

\newcommand{\one}{{\bf 1}}

\newcommand{\eid}{\stackrel{{\rm d}}{=}}

\newcommand{\betast}{ {\beta_\ast }}

\newtheorem{theorem}{Theorem}[section]
\newtheorem{proposition}[theorem]{Proposition}

\newtheorem{lemma}[theorem]{Lemma}
\newtheorem{example}[theorem]{Example}%
\newtheorem{remark}[theorem]{Remark}%
\newtheorem{definition}[theorem]{Definition}

\newtheorem{assumption}[theorem]{Assumption}

\title{Moderately Heavy Extreme Values under Extreme Long-Range Dependence}

\author*[1]{\fnm{Zao-Li} \sur{Chen}}\email{zaolichen@ustc.edu.cn}

\affil*[1]{\orgdiv{School of Mathematical Sciences}, \orgname{University of Science and Technology of China}, \orgaddress{\street{96 Jinzhai Road}, \city{Hefei}, \postcode{230026}, \state{Anhui}, \country{China}}}

\abstract{We consider stationary sequences whose marginal tail is subexponential and lies in the Gumbel Maximum domain of attraction. Due to the extremely strong dependence, their extreme values are caused by multiple big values and are clustered in the large scale with fractal features.   To characterize these delicate phenomena, we establish functional extremal limit theorems with non-Gumbel limits.  Our results contribute  to the unsolved problem in \cite{chen:samorodnitsky:2022:article}.}

\begin{document}

\numberwithin{equation}{section}

\maketitle

{
  \hypersetup{linkcolor=blue}
  \tableofcontents
}

\newpage

\section{Introduction}
\label{sec:introduction}

\subsection{Extreme Value Theory and Long Range Dependence}

We  investigatea class of stationary processes $(X_t:t\in \bbz)$ with subexponential marginal distributions belonging to the  Gumbel maximum domain of attraction (MDA). These processes are long range dependent (LRD) such that their extremal clusters exhibit fractal features at large scales. We contribute to understanding two fundamental questions: $(\rom1)$  how subexponential extreme values illuminate LRD, and, conversely, $(\rom2)$  how LRD accentuates the subtle distinctions among different families of subexponential distributions.

\vspace{1em}

Consider a random variable $X$ with a tail distribution $\overline{H}(x)=\PP(X > x)$. We say $X$ is \textit{subexponential} if
\begin{equation}
    \label{eq:subexponential-definition}
    \lim_{x\to \infty} \frac{\overline{H\ast H}(x)}{ \overline{H} (x) } = 2;
\end{equation}
see \cite{foss:korshunov:zachary:2013:book}. 
Typically, the tail $\overline{H}$ decays subexponentially. Let $(X_t^\dagger:t\in \bbz)$ be an i.i.d.\ sequence with a subexponential distribution $H$ and denote by $M_n^\dagger = \max\{X_0^\dagger,\ldots, X_n^\dagger \}, n\in \bbn$ the partial maxima. 
Under regularity conditions, suppose that
\begin{equation}
\label{eq:iid-extreme-value-theorem;1}
    \frac{M_n^\dagger - a_n^\dagger}{b_n^\dagger} \Rightarrow G,
\end{equation}
for some non-degenerate distribution $G$. Then, up to affine transformation $G(ax+b)$ for $a>0$ and $b\in \reals$, $G(x)=G_\gamma(x)=\exp ( - (1+\gamma x)^{-1/\gamma} )$ for some $\gamma\geq 0$. If $\gamma>0$, then $G_\gamma$ is a $\gamma$-Fr\'echet distribution; otherwise $\gamma =0$ and it reduces to the standard Gumbel distribution; see \cite{deHaan:ferreira:2006:book}.

 \vspace{1em}

 Modern extreme value theory (EVT) studies dependent processes, such as stationary sequences. In such settings, exceedance over high thresholds may appear in clusters due to serial dependence. For subexponential marginal distributions, extreme values tend to exceed high thresholds, and the shape of extremal clusters reflects the underlying dependence structure. A key concept to quantify extremal clustering is the \textit{extremal index} $\theta$. Under strong mixing and weak dependent conditions, the existence of $\theta \in (0,1]$ can be established, see \cite{mikosch:wintenberg:2024:book}. Roughly speaking, $1/\theta$ describe the mean of the reciprocal of the cluster size. Hence, the smaller the $\theta$, the stronger the dependence. It should be noted that 
the extremal index is a mixed consequence of both dependence structures and marginal distributions; e.g.\ the AR(1) process with regularly varying tails; see \cite{kulik:soulier:2020:book}.

\vspace{1em}

This paper focuses on stationary sequences with an vanishing extremal index $\theta=0$, where serial dependence persists over long ranges such that extremal clusters are of size  $O_P(n)$ among $\{X_0,\ldots, X_n\}$. Although not exhaustive,  a systematic framework to  generate LRD processes exists.  A typical stationary process $(X_t:t \in \bbz)$ has the representation
\begin{equation}
    \label{eq:long-range-dependence-and-zero-extremal-index}
    X_t = \int_{\mathbb{Y}} f\circ \tau^t (y) \, \mathcal{N}(dy),
\end{equation}
where $(\mathbb{Y},\mathcal{Y},\mu, \tau)$ is an ergodic and conservative dynamical system,  $f$ is an integrable function
and $\mathcal{N}$ is an infinitely divisible random measure.
The persistence of extremes inherit the recurrence behavior of the underlying flow $\{ \tau^t: t\in \bbz\}$. And the shapes of extremal clusters reflect the interaction between the flow and the noise. Such a framework \eqref{eq:long-range-dependence-and-zero-extremal-index} allows a wide combination of dependence structures and marginal distributions. To see other connections between dynamical systems and EVT, we also refer to \cite{lucarini:faranda:freitas:freitas:kuna:holland:nicol:todd:vaienti:2016:book}.

\vspace{1em}

We study a canonical case of the model \eqref{eq:long-range-dependence-and-zero-extremal-index}, whose flow is moderated by renewal epochs with infinite mean. It is associated with a natural memory parameter $\beta \in (0,1)$. A bigger $\beta$ corresponds to a stronger LRD. This model was first introduced in \cite{rosinski:samorodnitsky:1996:article} with respect to $\alpha$-stable integrals. Its extremal properties was fully characterized in \cite{samorodnitsky:wang:2019:article} with extensions to regularly varying margins.  The seminal paper \cite{samorodnitsky:wang:2019:article}  revealed  an infinite number of phase transition points $\beta\in \left\{ \frac{m}{m+1}: m\in \bbn \right\}$. When $\beta \in \left( \frac{m-1}{m}, \frac{m}{m+1} \right]$, a $m$-big-jump principle holds. Extremal clusters exhibit fractal-like behaviors, which can be explained as aggregation of independent $\beta$-stable regenerative sets. For further developments on multiple stochastic integrations, we refer to \cite{bai:kulik:wang:2024preprint}, \cite{bai:wang:2024:article} and \cite{bai:wang:2023:article}.

\vspace{1em}

This paper investigates marginal distributions in the Gumbel MDA. In particular, such subexponential tails of the form
\begin{equation}
    \label{eq:subexponential-exponentially-slowly-varying}
    \overline{H}(x) \sim c \cdot \exp \big(  S (x) \big), \quad c>0, \quad S \in \mathsf{RV}_0.
\end{equation}
On one hand, such tails are rapid varying, hence substantially lighter than regularly varying tails. On the other hand, their structural similarity in \eqref{eq:subexponential-exponentially-slowly-varying} prevents them from being excessively light.  This paradoxical nature leads to unusual extremal behaviors in LRD processes with such marginal tails. The work of \cite{chen:samorodnitsky:2022:article} first described these pictures in the zone $\beta\in (0,1/2)$. Namely, a
 similar time-changed representation for the limit extremal process hold in both \cite{lacaux:samorodnitsky:2014:article} and \cite{chen:samorodnitsky:2020:article},  yet a distinct 2-big-jump principle occurs in \cite{chen:samorodnitsky:2020:article}. This result indicates the inconsistency between tails of the form \eqref{eq:subexponential-exponentially-slowly-varying} across different MDAs.

\vspace{1em}

The goal of this paper is to extend  \cite{chen:samorodnitsky:2022:article} to the case $\beta \in \left( \frac{1}{2}, 1 \right)$. For $\beta \in \left( \frac{m}{m+1} , \frac{m+1}{m+2} \right)$, we show that non-Gumbel limit objects occur under an anomalous ``$(m+1)$-big-jump principle". Namely, extremes before taking the limit are caused by $m$ big jumps and an extra big jump of a smaller magnitude. However, limit objects are subject to the ``$m$-big-jump" principle similar to \cite{samorodnitsky:wang:2019:article}. We fully characterize this interesting phenomena by proving extremal limit theorems, namely the weak convergences of random sup-measures and extremal processes. Our work contributes to revealing delicate differences among different types of subexponential distributions in an EVT approach.

\subsection{Organization}

In Section \ref{sec:moderately-heavy-distributions}, we develop the theory of the moderately heavy distributions with connections to the Gumbel MDA.  Section \ref{sec:random-sup-measures-and-extremal-processes} focus on the limit random sup-measures and limit extremal processes. We also compare them with other closely related limit objects in the previous literature. In Section \ref{sec:DNRMC}, we build the conservative dynamical system over null-recurrent Markov chains. The dynamics of return times and their independent intersections are analyzed.  In Section \ref{sec:extremal-limit-theorems}, we introduce the long range dependent processes prove their extremal limit theorems.

\subsection{Notations and Symbols}

We use the following list of notations/symbols throughout the paper.
\begin{itemize}
    \item $L \triangleq R$ means that $L$ is defined as $R$, or $R$ is defined as $L$. 
    \item Since $(\infty, -\infty) =\emptyset$, we write $\inf \emptyset  = \min \emptyset = \infty$ 
    and $\sup \emptyset = \max \emptyset = -\infty$.
    \item An ordered vector is denoted by $(x_1,\ldots, x_d)_<$, i.e.\ $x_1< \ldots < x_d$.
    \item For any finite measure $\Psi$ on $\bbr$, denote by $\overline{\Psi}(x)=\Psi(x,\infty)$ the tail measure.  
    \item For an increasing function $J:\bbr \to \bbr$,  $J^\leftarrow (x) \triangleq \inf\{ r : J(r) \geq x \}$. 
    \item For random variables $X$ and $Y$, $X \leq_P Y$ if $\PP(X \leq x) \leq \PP(Y\leq x)$ for all $x\in \bbr$.
    \item Let $(a_n)$ and $(b_n)$ be two positive sequences. Then:
    \begin{itemize}
        \item $a_n = o(b_n)$ if $\lim_{n\to \infty} a_n / b_n = 0$ and $a_n \sim b_n$ if $\lim_{n\to \infty} a_n / b_n = 1$; 
        \item $a_n \lesssim b_n$ if  $\limsup_{n\to \infty} a_n / b_n < \infty$ and $a_n \gtrsim b_n$ if $\liminf_{n\to \infty} a_n / b_n >0$;
        \item  $a_n \asymp b_n$ if both $a_n \lesssim b_n$  and $a_n \gtrsim b_n$.
        \item  For functions $f,g:\bbr_+\to (0,\infty)$, we adopt the same asymptotic notations whose meanings are obvious.
    \end{itemize}

\end{itemize}


\section{Moderately Heavy Distributions}
\label{sec:moderately-heavy-distributions}

\subsection{Subexponential Tails}

\label{subsec:subexponential-tails}

 A useful criterion of subexponentiality is given below.

\vspace{1em}

\begin{proposition}[\cite{pitman:1980:article} 
Theorem $\Rom{2}$]
\label{prop:pitman-subexponentiality-criterion}
\phantom{blank}

     Let $X$ be a random variable that is unbounded on the right.  For all $x\geq 0$, set $g(x) = -\log \PP(X > x)$. Suppose that there is $x_0 >0$ such that 
    $g$ is differentiable on $(x_0,\infty)$ and $g^\prime$ decreases to $0$ as $x\to \infty$. If 
   $$
         \int_{x_0}^\infty \exp \big( x g^\prime (x) - g(x) \big) g^\prime (x) dx < \infty,
    $$
    then $X$ is subexponential. 
\end{proposition}

\vspace{1em}

Let the probability distribution $H$ be subexponential  satisfying \eqref{eq:subexponential-definition}. 
If $F$ is another finite measure (not necessarily a probability distribution) such that $\overline{F} \sim c \overline{H}$ for some constant $c>0$, then $F$ is also subexponential; see \cite{samorodnitsky:2016:book} Proposition 4.1.7.
In this paper, we shall construct  subexponential distributions from infinitely divisible distributions. The following proposition guarantees the feasibility.

\vspace{1em}

\begin{proposition}[\cite{embrechts:goldie:veraverbeke:1979:paper} Theorem 1]
\label{prop:subexponential-from-Levy-measure}
\phantom{blank}

     Let $X$ be an infinitely divisible distribution with the characteristic triple $(\sigma^2,\nu,b)$ where $\nu$ is the L\'evy measure. If the  distribution $\nu(1,x] / \nu (1,\infty), x\geq 1$ is subexponential, then $X$ is subexponential and 
    $
    \mathbb{P}(X > x) \sim \nu(x,\infty).
    $
\end{proposition}

\vspace{1em}

\subsection{Gumbel Maximum Domain of Attraction}
\label{subsec:gumbel-max-domain-of-attraction}

Structures of distributions that lie in the Gumbel MDA are well known. We review some useful and fundamental facts  from Chapter 0 and Chapter 1 in \cite{resnick:1987:book}. 

\vspace{1em}

 Let $\Lambda(x) = \exp ( - \exp (-x) )$ denote the standard Gumbel distribution function on $\bbr$. 
A distribution $H$ on $\bbr$ lies in the Gumbel MDA, i.e.\  $H \in \mathsf{MDA}(\Lambda)$, if  the tail $\overline{H}$ admits a \textit{von Mises representation}; see (1.3) in \cite{resnick:1987:book}. Namely, there exists a $x_0\in \bbr$ such that

\begin{equation}
        \label{eq:von-mises-representation-defintion}
        \overline{H} (x) = c(x) \exp \left(  - \int_{x_0}^x \frac{1}{h(u)} du  \right)  \quad \forall\, x \geq x_0;
    \end{equation}   
    \begin{itemize}
        \item[($\rom1$)] $c(\cdot)$ is a positive function on $(x_0,\infty)$ with $\lim_{x\to \infty} c(x)  \in (0,\infty)$;
        \item[($\rom2$)]  $h(\cdot)$ is the \textit{auxiliary function}, which is positive and absolutely continuous on $(x_0,\infty)$ such that
        $\lim_{x\to \infty} h^\prime (x) = 0$. Thus, $h(x)=o(x)$.
    \end{itemize}
For an i.i.d.\ sequence $(X(t): t\in \bbn)$ with the marginal distribution $H$ in \eqref{eq:von-mises-representation-defintion}, then one can take
\begin{equation}
\label{eq:def-gumbel-maximum-domain-of-attraction;1}
    b_n = (1/ \overline{H})^\leftarrow (n)  \quad   \text{and} \quad a_n =  h(b_n)
\end{equation}
so that 
\begin{equation}
\label{eq:def-gumbel-maximum-domain-of-attraction;2}
    \frac{\max\{ X(1),\ldots, X(n) \} - a_n}{b_n} \Rightarrow \Lambda.
\end{equation}
Below we characterize the  variability properties of $ 1 \big/  \overline{H} $ and $ (1 \big/ \overline{H})^\leftarrow \triangleq V_H$.

\vspace{1em}

\begin{proposition}[Exercise 1.1.9, Proposition 0.9 and Proposition 0.10 in \cite{resnick:1987:book}]
\label{proposition:Gamma-Pi-Variation}
\phantom{bk}
  
    Let the distribution $H\in \mathsf{MDA}(\Lambda)$ be in \eqref{eq:von-mises-representation-defintion}. Let $\overline{\Psi}$ be any tail measure on $\bbr$ such that 
    $ \overline{\Psi} (x) \sim c \overline{H}(x)$ for some constant  $c>0$.

    \begin{itemize}
          \item[$(\rom1)$] The tail $\overline{H}$ is rapid varying, formally $\overline{H} \in \mathsf{RV}_{-\infty}$ and
          \begin{equation}
         \label{eq:MDA-Lambda-rapid-varying}
        \lim_{t\to \infty} \frac{\overline{H}(tx)}{\overline{H}(t)} = 
        \begin{cases}
           0, \quad & x>1 \\
           \infty, & x \in (0,1)
        \end{cases}.
         \end{equation}
          \item[$(\rom2)$] The function $ 1 \big/  \overline{H} $ is $\Gamma$-varying up to tail equivalence, i.e.\ for any $X>0$,
           \begin{equation}
         \label{eq:Gamma-varying-definition}
        \lim_{t\to \infty} \frac{\overline{H}(t)}{\overline{H}(t+ x h(t))} = \lim_{t\to \infty} \frac{\overline{\Psi}(t)}{\overline{\Psi}(t+ x h(t))} = e^x.
     \end{equation}
           \item[$(\rom3)$] The  function $V_H$ is $\Pi$-varying,  i.e.\ for any $x>0$, 
            \begin{equation}
        \label{eq:Pi-varying-definition}
     \lim_{t\to \infty}   
     \frac{V_H(tx) - V_H(t)}{ h\circ V_H(t) }
     = \log  x  \quad  \text{and} \quad 
     \lim_{t\to \infty}   
     \frac{V_H(t) - V_\Psi(t)}{ h\circ V_H(t) } = -\log c.
    \end{equation}
    \end{itemize}
\end{proposition}

\subsection{Moderately Heavy Tails}

\subsubsection{Definition and Examples}

\begin{definition}[Moderately Heavy Distribution]
\label{def:exp-slow-vary-distr}
\phantom{blank}

    Let $H\in \mathsf{MDA}(\Lambda)$  with the
 von Mises representation
\begin{equation} \label{eq:marginal-tail-distribution-von-mises-representation}
    \overline{H}(x) \sim c_\# \overline{H_\#} (x), \quad \overline{H_\#}(x) =  \exp \left( - \int_{x_0}^x \frac{du}{h(u)}   \right) , \quad x\geq x_0
\end{equation}
for some constants $c_\# >0$ and $x_0\in \reals$.
  We call the distribution $H$ (or its tail $\overline{H}$)  moderately heavy if the following conditions hold.
\begin{itemize}
    \item[$(\rom1)$]  $h \in \mathsf{RV}_1$.   
    \item[$(\rom2)$] Set $ V_\#  \triangleq ( 1\big/  \overline{H_\#} )^\leftarrow \in \mathsf{RV}_0$. Then, for some $v_0, z_0>0$, $V_\#$  satisfies 
\begin{equation}
    \label{eq:slowly-varying-subexp-quantile}
    V_\#(z) \sim v_0 \cdot \exp \left(  \int_{z_0}^z \frac{\zeta(\log u)}{u \log u} du   \right),
    \quad    V_\# ^\prime (z)  \asymp \left[ \exp \left(  \int_{z_0}^z \frac{\zeta(\log u)}{u \log u} du   \right)\right]^\prime.
\end{equation}
where $v_0>0$ and $\zeta \in \mathsf{RV}_{\alpha}$ such that $\alpha\in [0,1)$ and, for some $\delta>0$,
\begin{equation}
    \label{eq:growth-condition-of-zeta(u)}
    \frac{\zeta(u)}{(\log u)^\delta} \to \infty \quad  \text{and} \quad 
    \zeta (u)  \lesssim \frac{u}{\log u}.
\end{equation}

\end{itemize}

\end{definition}

\vspace{1em}

\begin{remark}[Motivation of Definition \ref{def:exp-slow-vary-distr}]
\label{remark:moderately-slowly-varying}
\phantom{blank}
{\rm
\begin{itemize}
    \item By the Karamata's integral criterion, the above condition $(\rom1)$ forces 
    $$
    S(x) = \int_{x_0}^x \frac{du}{h(u)}  \in \mathsf{RV}_0,
    $$
    which agrees with the informal definition \eqref{eq:subexponential-exponentially-slowly-varying}.
    \item Since $1/\overline{H_\#} \in \mathsf{RV}_\infty$, the function $V_\# \in \mathsf{RV}_0$ by Proposition 0.8 (v) in \cite{resnick:1987:book}. The choice of $\alpha \in [0,1)$ agrees with the Karamata representation of $\mathsf{RV}_0$ functions. We also note that 
    $$
    \frac{1}{\overline{H_\#} (x)} = \exp \circ I_h (x), \quad I_h(x) \triangleq \int_{x_0}^x \frac{du}{h(u)}.  
    $$
    From \cite{embrechts:hofert:2014:notes} Proposition 2.3 (8), we get that
    $$
    \big( 1/\overline{H_\#} \big)^\leftarrow (z) = I_h^\leftarrow \circ \log z .
    $$
    Since $I_h \in \mathsf{RV}_0$, we get $I_h^\leftarrow\in \mathsf{RV}_\infty$. The  right-hand-side of \eqref{eq:slowly-varying-subexp-quantile} is 
    $$
    v_0 \cdot \exp \left( \int_{\log z_0} ^ {\log z} \frac{\zeta(u)}{u} du \right),
    $$
    which is a composition of a $\mathsf{RV}_\infty$ function with $\log z$.
    \item Definition \ref{def:exp-slow-vary-distr} is a much neater replacement of Assumption 3.1 in \cite{chen:samorodnitsky:2022:article}. Namely, 
    the condition $\zeta \in \mathsf{RV}_a$ is sufficient for the conditions (B1), (B3) and (B4) ibid. The verification is direct.
\end{itemize}
}
\end{remark}

\vspace{1em}

\begin{example}[Log-Normal Tails]
\label{example:log-normal-tail}
\phantom{blank}

   For any parameter $\gamma > 1$,  define the  tail  probability
    \begin{equation*}
         \overline{H_\#} (x)=
        \exp\left( - (\log x )^{ \gamma} \right) , \quad x \in (1,\infty).
    \end{equation*}
    \begin{itemize}
        \item The auxiliary function is
        \begin{equation*}
            h(u) = \frac{u}{\gamma (\log u)^{\gamma - 1}}\in \mathsf{RV}_1,
        \end{equation*}
        which satisfies $h^\prime (u) > 0$ for $u > 1$.
        \item The quantile function and its Karamata representation are
        \begin{align*}
            & V_\#(x) = \exp \left(  (\log x )^{\frac{1}{\gamma}}  \right), \quad x \geq 1 ; 
            \\
            & V_\#(x) = \exp\left( \int_1^x \frac{(\log u)^{\frac{1}{\gamma}-1}}{\gamma u}  du \right), \quad \zeta (u) = \frac{1}{\gamma} u^{\frac{1}{\gamma}}\in \mathsf{RV}_{\frac{1}{\gamma}}.
        \end{align*}
    \end{itemize}
This example corresponds to  $\zeta\in \mathsf{RV}_\alpha, \alpha =  1 / \gamma \in (0,1)$.
\end{example}

\vspace{1em}

\begin{example}[Super-Log-Normal Tails]
\label{example:super-log-normal-tail}
\phantom{blank}

Let $k\in \bbn$, denote the $k$-times composed exponential and logarithm by
\begin{equation*}
    E_k (x)= \underbrace{ \exp \circ \cdots \circ \exp}_{k\text{ compositions}} (x), \quad 
    L_k (x)= \underbrace{ \log \circ \cdots \circ \log}_{k\text{ compositions}} (x).
\end{equation*}
For parameters $\gamma \in (0,1)$ and $n\in \bbn$,  define for $x>E_n (0)$,
\begin{equation*}
    S_{n,\gamma} (x) = E_n 
\Big[ \big(  L_n (x) \big)^\gamma \Big ] \in \mathsf{RV}_0
\quad 
\text{and}
\quad
\overline{H_\#}(x) = 
        \exp\left( - S_{n,\gamma} (x) \right).
\end{equation*}
Though $\overline{H_\#} \to 0$, 
the measure $H_\#$ is obviously not a probability distribution.  
\begin{itemize}
    \item  The auxiliary function is 
\begin{equation*}
    h(u) = 
     \frac
     { u \cdot 
      \Big( L_n ( u ) \Big)^{1-\gamma}  \cdot 
      \prod_{k=1}^{n-1} L_k (u)
    }
    {  \gamma \cdot 
     \prod_{k=1}^n  E_k
    \bigg[  \Big( L_n (u)  \Big)^\gamma   \bigg]
    } \in \mathsf{RV}_1 , \quad x > E_n (0).
\end{equation*}
Clearly, $h^\prime (x) > 0$ for $x \gg 0$.
    \item The quantile function and the Karamata representation are 
    \begin{align*}
        & V_\#(z) =   E_n
\bigg[ \Big( L_{n+1} (z) \Big)^ \frac{1}{\gamma} \bigg ] = E_n(0) \cdot \exp \left( \int_{z_0}^z \frac{\zeta(\log u)}{u \log u }  du \right),  \quad z \geq z_0=E_{n+1} (0);
\\
 & \zeta( u ) = \frac{1}{\gamma} \cdot \frac{
        \Big( L_{n} (u) \Big )^{\frac{1}{\gamma} - 1}  
        \prod_{k=1}^{n-1} E_k \bigg[ \Big(  L_n (u)  \Big)^{\frac{1}{\gamma}}   \bigg]
        }
        { \prod_{k=1}^{n-1} L_k (u) } \in \mathsf{RV}_0.
    \end{align*}
\end{itemize}

This example illustrate the case  $\zeta\in \mathsf{RV}_0$. It also extends the Example 3.2 in
 \cite{chen:samorodnitsky:2022:article},  whose canonical form is $\overline{H_\#}(x) = \exp (- S_{1,\gamma})$. We also note the following asymptotic relations, which  helps to compare the heaviness of different tails in this family.
 \begin{align*}
     & S_{n,\gamma_1} = o (   S_{n,\gamma_2}  ), \quad \forall \, n \in \bbn, 0<
     \gamma_1 < \gamma_2 < 1; \\
     & S_{n_1,\gamma_1} = o (   S_{n_2,\gamma_2}  )
     , \quad \forall \, n_1, n_2 \in \bbn, n_1> n_2, \forall \, \gamma_1, \gamma_2 \in (0,1).
 \end{align*}
\end{example}

\subsubsection{Properties}

\begin{proposition}
\label{prop:slowly-varying-subexp-implies-subexponentiality}
\phantom{blank}

     The distribution 
     $H$ in Definition \ref{def:exp-slow-vary-distr}  is subexponential.
\end{proposition}

\vspace{1em}

\begin{proof}[Proof of Proposition \ref{prop:slowly-varying-subexp-implies-subexponentiality}] 
\phantom{blank}

   We verify Pitman's criterion in Proposition \ref{prop:pitman-subexponentiality-criterion}. For all sufficiently large $x$, we note that
   $$
       g(x) = - \log \overline{H_\#} = \int_{x_0}^x \frac{du}{h(u)} 
       \quad \text{and} \quad 
       g^\prime (x) = \frac{1}{h(x)} \downarrow 0
   $$
because of $(\rom1)$ and $(\rom2)$ in Definition \ref{def:exp-slow-vary-distr}. Next, we use again that
 $h\in \text{RV}_1$ and the integral Karamata's theorem to claim that  
$$
x g^\prime (x) =\frac{x}{h(x)} = o \left(  \int_{x_0}^x \frac{du}{h(u)}  \right) = o \big (g(x)  \big ).
$$
Thus, for all $x \gg 0$, 
$$
 x g^\prime (x) - g(x) \leq -\frac{1}{2} g (x) = - \frac{1}{2} \int_{x_0}^x \frac{du}{h(u)}. 
$$
Since $h(x) = o(x)$,   we may assume that $h(x) <  x/4$ for all $x\gg 0$. Thus, for some constant $C>0$, it follows that 
$$
- \frac{1}{2} \int_{x_0}^x \frac{du}{h(u)} 
\leq  - \frac{1}{2} \int_{x_0}^x \frac{du}{u/4} + C = 
 -2 \log x + C
$$
for some constant $C>0$. Thus, the Pitman's criterion holds because  
$$
\int_{x_0}^\infty \exp\big( xg^\prime (x) - g(x) \big) d x  
\lesssim \int_{x_0}^\infty \frac{dx}{x^2} < \infty.
$$
\end{proof}

\vspace{1em}

Some properties of the  function $V_\#$ are extremely important in the sequel. We summarize them into the Proposition \ref{prop:properties-exp-slowly-varying-distr}. A few technical manipulations are placed ahead in the Lemma \ref{lemma:properties-exp-slowly-varying-distr}.

\vspace{1em}

\begin{lemma}
    \label{lemma:properties-exp-slowly-varying-distr}
\phantom{blank}

    Let $H$ be  as  in Definition \ref{def:exp-slow-vary-distr} and we follow the same notations therein. 
    \begin{itemize}
        \item[$(\rom1)$] For $x > x_0$,  $h \circ V_\# (x) =  x V_\# ^\prime (x)$.
        \item[$(\rom2)$] Let the function $L(x) \to \infty$ be  such that 
        \begin{equation}
        \label{eq:pertubation-growth-condition}
            \log L (x) \lesssim \frac{\log x}{ \zeta (\log x )}.
        \end{equation}
        Then, the function $\zeta(z)$ over $z\in[\log x , \log x +\log L(x)]$ is roughly constant, i.e.,
        \begin{equation}
            \label{eq:zeta-is-roughly-constant}
             \sup_{z \in [\log x, \log x + \log L(x)]} \zeta(z) \sim   \inf_{z \in [\log x, \log x + \log L(x)]}  \zeta(z) \sim \zeta(\log x) .
        \end{equation}
        
    \end{itemize}
\end{lemma}

\begin{proof}[Proof of Lemma \ref{lemma:properties-exp-slowly-varying-distr}]
\phantom{blank}

\textbf{Part (i).}
We note that that $\overline{H_\#}( V_\#(x) ) = 1/x$, therefore 
$$
 \log x  = \int_{x_0}^{V_\#(x)} \frac{du}{h(u)}.
 $$
The Inverse Function Theorem implies that $V_\#$ is differentiable. 
Taking derivatives on both sides yields the identity  $h \circ V_\# (x) =  x V_\# ^\prime (x)$.

\vspace{1em}

\textbf{Part (ii).} It suffices to show that 
$$
\frac{\zeta(\log x + z)}{\zeta(\log x)} \sim 1 \quad 
\text{uniformly for all } z \in [\log x , \log x +\log L(x)].
$$
Since $\zeta \in \mathsf{RV}_\alpha$, then $\zeta$ admits a Karamata representation
$$
\zeta(z)= C(z) \exp \left( \int_{z_0}^z \frac{\epsilon(u)}{u} du \right)
$$
with $\lim_{z\to \infty} C(z) \in (0,\infty) $ and $\lim_{z\to \infty}\epsilon(z) = \alpha$. Obviously, 
$$
\sup_{z \in [\log x , \log x +\log L(x)]}
\left| 
\frac{C(\log x + z)}{C(\log x)} - 1 
\right| 
\to 0.
$$
We are left to deal with the integral part.

\vspace{1em}

To this end, we note that there exists $c_1<0$ and $ c_2 >0$ such that $\epsilon(z) \in [c_1,c_2]$ for all $z\gg 0$. It follows that 
\begin{align*}
   & \sup_{z\in [\log x , \log x +\log L(x)]} \exp\left( \int_{\log x}^{z} \frac{\epsilon(u)}{u} du \right) \\
    \leq & \exp\left(  c_2 \int_{\log x}^{\log x + \log L(x)} \frac{1}{u} du \right) 
    =  \exp \left[ c_2 \log \left( \frac{\log x + \log L(x)}{\log x} \right) \right] \xlongrightarrow{x\to \infty} 1 .
\end{align*}
The limit is one because 
$$
\log \left( \frac{\log x + \log L(x)}{\log x} \right) \sim \frac{\log L(x)}{\log x } \lesssim \frac{1}{\zeta(\log x)} \to 0.
$$
Similarly, we obtain 
\begin{align*}
    \inf_{z\in [\log x , \log x +\log L(x)]} \exp\left( \int_{\log x}^{z} \frac{\epsilon(u)}{u} du \right) & \geq  \exp \left[ c_1 \log \left( \frac{\log x + \log L(x)}{\log x} \right) \right] \\
   & \xlongrightarrow{x\to \infty} 1 ,
\end{align*}
which finishes the proof.
\end{proof}

\vspace{1em}

\begin{proposition}
    \label{prop:properties-exp-slowly-varying-distr}
\phantom{blank}

     Let the  setup be the same as Lemma \ref{lemma:properties-exp-slowly-varying-distr}.
     
    \begin{itemize}
        \item[$(\rom1)$]  For the $\delta>0$ in Definitition \ref{def:exp-slow-vary-distr} $(\rom3.2)$,  
        $$
        V_\#(x)  \gtrsim  \exp\Big(   ( \log \log x)^{1+\delta} \big/ (1+\delta)\Big)  \to \infty.
        $$
        Hence, for any $c>0$, 
        $$
        \frac{ V_\#(x)}{(\log x)^c} \to \infty \quad \text{and} \quad \frac{h \circ V_\#(x)}{ (\log x)^c} \to \infty.
        $$
        
        \item[$(\rom2)$]  Let $\alpha_1 > \alpha_2 > 0$. Then, for the $\delta>0$ in Definitition \ref{def:exp-slow-vary-distr} condition $(\rom2)$
        and any $0<c < (1+\delta)^{-1} \log (\alpha_1/\alpha_2)$, 
        \begin{equation}
            \label{eq:quantile-ratio-lower-bound}
              \frac{V_\#(x^{\alpha_1})}{V_\#(x^{\alpha_2})} \asymp 
         \frac{h \circ V_\#(x^{\alpha_1})}{h \circ V_\#(x^{\alpha_2})}  \quad \text{and} \quad
          \frac{V_\#(x^{\alpha_1})}{V_\#(x^{\alpha_2})} 
         > \exp\big(  c (\log \log x)^\delta \big)
        \end{equation}
       
        \item[$(\rom3)$] Let $L(x) \to \infty$ satisfy \eqref{eq:pertubation-growth-condition}. Then, 
         \begin{equation}
           \limsup_{x\to \infty} \frac{V_\#(x L(x) )}{V_\#(x)} < \infty
           \quad \text{and} \quad 
            \frac{V_\#(x L(x) ) - V_\#(x) }{  h\circ V_\#(x) } \asymp \log L(x) . 
            \label{eq:pertubation-result;2.1}
        \end{equation}
    \end{itemize}
\end{proposition}

\begin{proof}[Proof of Proposition \ref{prop:properties-exp-slowly-varying-distr}]
\phantom{blank}

\textbf{Part ($\rom{1}$).}  For all $x \gg 0$, it follows that 
\begin{align*}
    V_\# (x) 
    \gtrsim \exp \left(  \int_ e ^ x \frac{(\log \log u)^{\delta}}{u \log u} du \right) \gtrsim  \exp \left(  \frac{1}{1+\delta} (\log \log x)^{1+\delta}  \right). 
\end{align*}
Thus, $V_\# (x) / (\log x)^c \to \infty$ for any $c>0$. So is $h\circ V_\# (x)$ because $h \in \mathsf{RV}_1$.

\vspace{1em}

\textbf{Part ($\rom2$).} 
By Lemma \ref{lemma:properties-exp-slowly-varying-distr} $(\rom1)$ , it follows that 
$$
\frac{h\circ V_\#(x^{\alpha_1})}{h\circ V_\#(x^{\alpha_2})} 
= \frac{V_\#(x^{\alpha_1})}{ V_\#(x^{\alpha_2})} \cdot \frac{\zeta(\alpha_1 \log x)}{\zeta(\alpha_2 \log x)}  \cdot \frac{\alpha_2 \log x}{\alpha_1 \log x} \asymp \frac{V_\#(x^{\alpha_1})}{ V_\#(x^{\alpha_2})}
$$
because $\zeta$ is regularly varying. For $x\gg 0$,  Definition \ref{def:exp-slow-vary-distr} $(\rom2)$ implies that
\begin{align*}
      \frac{V_\#(x^{\alpha_1})}{V_\#(x^{\alpha_2})}  \sim &
      \exp 
\left(
\int_{ x^{\alpha_2} } ^ { x^{\alpha_1}} \frac{\zeta(\log u)}{u \log u} du  
\right) \\
\gtrsim & 
 \exp \left(  \frac{1}{1+\delta} (\log \log x)^{1+\delta}  \, \bigg|_{ x^{\alpha_2} } ^ { x^{\alpha_1}} \right)
 \geq \exp \left(  c (\log \log x)^{\delta}  \right).
\end{align*}

\vspace{1em}

\textbf{Part ($\rom3$).}
For the first part of \eqref{eq:pertubation-result;2.1},
 note that 
\begin{equation*}
    \frac{V_\#(x L(x) )}{V_\#(x)}  
    \sim \exp \left(  
   \int_{\log x} ^ { \log x + \log L (x) } \frac{\zeta(u)}{u} du 
    \right), \quad x \gg 0.
\end{equation*}
By \eqref{eq:zeta-is-roughly-constant}, it follows that 
\begin{equation*}
    \int_{\log x} ^ { \log x + \log L (x) } \frac{\zeta(u)}{u} du 
\sim    \zeta(\log x ) \int_{\log x} ^ { \log x + \log L (x) } \frac{du}{u} 
\sim  \frac{\zeta(\log x) \log L(x)}{ \log x} = O(1),
\end{equation*}
and, hence, $V_\#(x L(x) ) \asymp V_\#(x)$.

\vspace{1em}

For the second part of \eqref{eq:pertubation-result;2.1},
apply first Lemma \ref{lemma:properties-exp-slowly-varying-distr} $(\rom1)$ to derive that
$$
V_\#(xL(x)) - V_\#(x) = \int_{x} ^ {x L(x)} V^\prime_\#(u) du = 
 \int_{x} ^ {x L(x)} \frac{h\circ V_\#(u) }{u}.
$$
The argument of $V_\#(x L(x) ) \asymp V_\#(x)$ obviously leads to 
$$
\sup_{u \in [x, xL(x)]} V_\# (u )   \asymp \inf_{u \in [x, xL(x)]} V_\#( u )  \asymp V_\# (x).
$$
Since $h\in \mathsf{RV}_1$ by Definition \ref{def:exp-slow-vary-distr} $(\rom2)$, then we immediately obtain that 
$$
\frac{V_\#(xL(x)) - V_\#(x)}{h\circ V_\#(x)} 
 \asymp \int_x^{xL(x)} \frac{du}{u}  = \log L(x).
$$

\end{proof}


\section{Random Sup-Measures}
\label{sec:random-sup-measures-and-extremal-processes}

\subsection{Random Closed Sets}
\label{subsec:random-closed-sets}

\subsubsection{Basics}

Denote by $E$ the topological space $[0,1]$ or $[0,\infty)$.  
We use the notation $\mathcal{B} =\mathcal{B} (E) $, $\mathcal{G} =\mathcal{G} (E) $, $\mathcal{F} = \mathcal{F}(E)$ and $\mathcal{K}= \mathcal{K}(E)$ for the families of Borel, open, closed and compact sets of $E$, respectively. For either choice of $E$, the \textit{Fell} topology $\mathcal{B}(\mathcal{F})$ is generated by the subbasis
\begin{align*}
& \mathcal{F}_G:=  \{ F \in \mathcal{F}: F \cap G \neq \emptyset  \}, \quad G\in \mathcal{G}; \\
& \mathcal{F}^K : =  \{ F \in \mathcal{G}: F \cap K = \emptyset \} , \quad 
K \in \mathcal{K}.
\end{align*}
The topological space $(\mathcal{F}, \mathcal{B}(\mathcal{F}))$ is metrizable and compact, see \cite{molchanov:2017:TRS-book} \S 1.1.1. 

\vspace{1em}

A \textit{random closed set} $F$ is a measurable mapping  from a probability space to  $(\mathcal{F}, \mathcal{B}(\mathcal{F}))$. For $E=[0,\infty)$, we say $R$ is \textit{stationary} if $$ (F-r) \cap [0,\infty)  \eid F, \quad \forall \, r \geq 0.  $$    
Let $\mathfrak{I}$ be the collection of all finite unions of open intervals and $\mathfrak{C}_F$ be the collection of all continuity sets of $F$, i.e.\
\begin{align*}
   \mathfrak{C}_F = \big \{ B \in \mathcal{B}: & \, \text{closure } B \in \mathcal{K}, \\
   & \, \PP ( F \cap  \text{closure } B \neq \emptyset) = \PP ( F \cap  \text{interior } B \neq \emptyset)
   \big \}. 
\end{align*}
A sequence of random closed sets $(F_n:n\in \bbn)$ weakly converges to $F$ if 
\begin{equation}
    \label{eq:weak-convergence-of-random-closed-set}
    \PP (F_n \cap I \neq \emptyset ) \to 
    \PP(F \cap I \neq \emptyset ), 
    \quad \forall \, I \in \mathfrak{I} \cap \mathfrak{C}_F.
\end{equation}

\subsubsection{Stable Subordinators and Stable Regenerative Sets}
\label{subsubsec:subordinator-and-regenerative-sets}

We introduce an important family of random closed sets on $E=[0,\infty)$.
For each $\beta \in (0,1)$, let $( \mathscr{Y}_\beta (t): t \in \reals_+ )$ denote the \textit{standard} $\beta$-\textit{stable subordinator}. It is an strictly increasing L\'evy process such that the margin  
$\mathscr{Y}_\beta (1)$ has the $S_\beta(1,1,0)$ distribution,  whose Laplace transform is
$$
\EE e^{- \gamma \mathscr{Y}_\beta (1) } = \exp \left(  
-  \frac{\gamma^\alpha}{\cos (\pi \beta / 2)} \right),
$$
see \cite{samorodnitsky:taqqu:1994:book} Proposition 1.2.12.
Denote by $\mathsf{Range} ( \mathscr{Y}_\beta) = \{   \mathscr{Y}_\beta (t): t \in \reals_+    \}$ the range of the subordinator. Its closure 
\begin{equation}
    \label{eq:stable-regenerative-set-positive-real-axis}
    \overline{ \mathsf{Range}( \mathscr{Y}_\beta)  } = \mathsf{closure} (\mathsf{Range}( \mathscr{Y}_\beta))
\end{equation}
is thus a random closed set. This is the $\beta$-\textit{stable regenerative set}  on the positive real axis $\reals_+$.
Trivially, $0 \in \overline{ \mathsf{Range}(\mathscr{Y}_\beta)  }$ with probability one. In addition, almost surely,  
\begin{itemize}
    \item it is scale invariant, i.e.\ $c\overline{ \mathsf{Range}( \mathscr{Y}_\beta)  } \eid \overline{ \mathsf{Range}( \mathscr{Y}_\beta)  }$ for any $c >0$,
    \item its Hausdorff dimension equal to $\beta$,
\end{itemize}
see \S \Rom{3}.5 in \cite{bertoin:1996:book}. 

\vspace{1em}

In the sequel, we shall also need the inverse process of $\beta$-stable subordinator,
\begin{equation}
    \label{eq:Mittag-Leffler-process}
    \mathscr{Z}_\beta (t) \triangleq  \mathscr{Y}_\beta ^\leftarrow (t), \quad t\in \reals_+.
\end{equation}
This process $(\mathscr{Z}_\beta(t):t \in \reals_+)$ is called a \textit{Mittag-Leffler} process, whose self-similar index is $\beta$. Furthermore, because stable subordinators are strictly increasing, $(\mathscr{Z}_\beta(t):t \in \reals_+)$ is therefore a.s.\ continuous. 

\subsection{Random Sup-Measures}
\label{subsec:random-sup-measures}

\subsubsection{Basics}
\label{subsubsec:basics-random-sup-measures}

We follow the notations in \S \ref{subsec:random-closed-sets} to introduce (random) sup-measures. These notions were first studied in 
\cite{obrien：torfs:vervaat:1990:article} and are very useful in the extreme value theory of stationary processes, see  \cite{molchanov:strokorb:2016:article}, \cite{wang:2022:article} and the literatures therein. We adapt our materials from \cite{chen:samorodnitsky:2022:article} and Appendix G in \cite{molchanov:2017:TRS-book}.


\paragraph{Sup-Measures}

A \textit{sup-measure} $m$ on $E$ is a measurable mapping $m:\mathcal{G} \to \overline{\reals}=[-\infty,\infty]$ such that 
$$
m(\emptyset) = - \infty \quad \text{and} \quad 
m\left( \bigcup_{\alpha \in \mathcal{A}} G_\alpha  \right) = \sup_{\alpha \in \mathcal{A}} m (G_\alpha)
$$
for any arbitrary collection $(G_\alpha: \alpha \in \mathcal{A})$ of open sets. 
The \textit{sup-derivative} $d^\vee m: E \to \overline{\reals}$ is given by 
$
d^\vee m (t) = \inf_{G \ni t} m (G),
$
which is automatically an upper semi-continuous function. Conversely, for any $\overline{\reals}$-valued function $f$, the \textit{sup-integral} $i^\vee f$ is defined as 
$
i^\vee f (G) = \sup_{t\in G} f(t),
$
which is always a sup-measure. 
    Because $m= i^\vee d^\vee m$, we can thus extend the domain of $m$ from $\mathcal{G}$ to $\mathcal{B}$ through
    $$
    m(B) = \sup_{t \in B} d^\vee m (t), \quad \forall \, B \in \mathcal{B}.
    $$

\vspace{1em}

Let $\mathsf{SM}=\mathsf{SM}(E)$ be the collection of all sup-measures on $E$. A sequence of sup-measures $m_n, n\in \bbn$ converges to a sup-measure $m$ in the \textit{sup-vague} topology if 
\begin{align*}
 \limsup_{n\to \infty} m_n (K) \leq m(K) , \; \forall \, K \in \mathcal{K}  \quad \text{and} \quad 
   \liminf_{n\to \infty} m_n (G) \geq m(G) , \; \forall \, G \in \mathcal{G}.
\end{align*}
 The space $\mathsf{SM}$, Under the sup-vague topology, is compact and metrizable. 

\paragraph{Random Sup-Measures}

A \textit{random sup-measure} $M$ is a measurable mapping from a probability space to $\mathsf{SM}$.  For $E=[0,\infty)$, a random sup-measure $M$ on $E$ is \textit{stationary} if 
$$
M (r + \cdot) \eid M(\cdot), \quad \forall \, r > 0.
$$
Regardless of the choice of $E$, 
the family $\mathscr{I}(M)$ of \textit{continuity intervals} is defined as 
$$
\mathscr{I}(M) = \big\{ I \subset E  \text{ is an open interval such that } M(I)  \xlongequal{a.s.} M\big(\mathsf{closure}(I) \big)  \big \}
$$
A sequence of random sup-measures $M_n$ converges weakly to  $M_\infty$ if and only if
\begin{equation}
    \label{eq:weak-convergence-criterion-continuity-intervals}
    \big(M_n (I_1), \ldots, M_n(I_k)   \big ) \Rightarrow \big( M_\infty (I_1), \ldots, M_\infty(I_k)   \big )
\end{equation}
for any arbitrary (finite) collection of disjoint open intervals $I_1,\ldots, I_k \in \mathscr{I}(M)$.

\subsubsection{Limit Random Sup-Measures}
\label{subsubsec:limit-random-sup-measure}

\paragraph{Construction and Properties}

For all $\beta \in (0,1)$, consider a Poisson point process $\mathscr{P}$ on $\reals \times \reals_+ \times \mathcal{F}(\reals_+)$ with intensity measure 
\begin{equation}
    \label{eq:limit-rsm-poissonPP-intensity-measure}
    e^{-u} du \times (1-\beta) q^{-\beta} \times dP_\beta,
\end{equation}
where $P_\beta$ is the law of the stable regenerative set  in \eqref{eq:stable-regenerative-set-positive-real-axis}. 
Let $(U_j, Q_j, Z_j)$ be a measurable enumeration of the Poisson point process $\mathscr{P}$. Denote by 
$$
W_j= Q_j + Z_j, \quad j \in \bbn
$$
shifted stable regenerative sets. When we consider finite many $W_j$'s and their intersection,  a nice zero-one dichotomy holds.

\vspace{1em}

\begin{lemma}[Lemma 3.1 and (3.5) in  \cite{samorodnitsky:wang:2019:article}]
    \label{lem:zero-one-dichotomy-srs-intersections}
    \phantom{blank}

\begin{equation}
    \label{eq:zero-one-dichotomy-srs-intersections}
    \PP\left( \bigcap_{i=1}^{s} W_{i}   \neq \emptyset \right) = 
    \begin{cases}
    1,  \quad & s \leq \ell_\beta \\
     0,   & s > \ell_\beta 
    \end{cases}, \quad \ell_\beta =  \max\left \{  \ell \in \bbn : \ell < \frac{1}{1-\beta} \right \}.
\end{equation}
Whenever $s \leq \ell_\beta$, the intersection $ \bigcap_{i=1}^{s} W_i $ is a randomly shifted $(s \beta -s + 1)$-stable regenerative set.
\end{lemma}

\vspace{1em}

\begin{proposition}[Limit Random Sup-Measure]
\label{proposition:limit-rsm-stationary-self-affine}
\phantom{blank}

\begin{itemize}
    \item[$(\rom1)$] Define
 a random  function $\eta:\reals_+ \to \overline{\reals}$ as
\begin{equation}
    \label{eq:limit-rsm-sup-derivative-eta}
    \eta (t) = 
\begin{cases}
    \sum_{j=1}^\infty U_j  \one(t\in W_j), \quad & \text{if } \sum_{j=1}^\infty   \one(t\in W_j) = \ell_\beta \\
    -\infty , &\text{otherwise}
\end{cases}.
\end{equation}
      Then, a.s.,  $\eta$  is  well-defined and upper semi-continuous.

   \item[$(\rom2)$] Let $ \mathcal{M}_\beta$ be the random sup-measure given by 
   \begin{equation}
    \label{eq:limit-rsm-on-half-line}
    \mathcal{M}_\beta(B) = \sup_{t\in B} \eta (t), \quad B \in \mathcal{B}(\reals_+).
\end{equation} 
Then,
\begin{itemize}
     \item[$(\rom2.1)$] $ \mathcal{M}_\beta$ is stationary, $\mathcal{M}(\cdot + r) \eid \mathcal{M}(\cdot)$ for any $r\geq 0$;
    \item[$(\rom2.2)$] $ \mathcal{M}_\beta$ is self-affine, $\mathcal{M}(a \cdot) \eid \mathcal{M}(\cdot)+ (1-\beta) \log a$ for any  $a > 0$.
\end{itemize}

\end{itemize}

\end{proposition}

\vspace{1em}

\begin{proof}[Proof of Proposition \ref{proposition:limit-rsm-stationary-self-affine}]
    \phantom{blank}

      \textbf{Part (i).}  Note that $\int_0^\infty e^{-u}du =1$ and $\int{-\infty}^0 e^{-u}du =\infty$, thus $U_j \stackrel{a.s.}{\longrightarrow} -\infty$. It follows that 
       $\eta$ is well-defined because each $t$ belongs to at most $\ell_\beta$ many random closed sets $W_j$'s, see \eqref{eq:zero-one-dichotomy-srs-intersections}. And, almost surely, the collection 
      $$
      \left\{  (j_1,\ldots, j_{\ell_\beta})_< \in \bbn^{\ell_\beta} : \bigcap_{i=1}^{\ell_\beta} W_{j_i} \neq \emptyset   , \sum_{i=1}^{\ell_\beta} U_{j_i} \geq a \right\}
      $$
      is finite for each $a\in \overline{\reals}$. Hence,  $\eta$ is upper semi-continuous because $\{ t\in \reals : \eta(t) \geq a \}$ is closed.

     \vspace{1em}

      \textbf{Part (ii).}  See (2.35) and (2.36) in \cite{chen:samorodnitsky:2022:article}.
  
\end{proof}

\paragraph{An Restricted Representation}

Given a stationary sequence $(X_t: t \in \bbz)$, it naturally induces a family of random sup-measures $(\mathcal{M}_n: n\in \bbn)$ as 
\begin{equation}
   \label{eq:stationary-sequence-empirical-rsm}
   \mathcal{M}_n (B) = \max_{t\in nB} X_t, \quad B \in \mathcal{B}(\reals).
\end{equation}
Assume that we have the weak convergence 
\begin{equation}
    \label{eq:weak-convergence-empirical-rsm}
    \frac{\mathcal{M}_n- b_n}{a_n} \Rightarrow \mathcal{M}_\infty
\end{equation}
for some normalization constants $a_n, b_n$ and limit random sup-measure $\mathcal{M}_\infty$. Then,  $\mathcal{M}_\infty$ must be stationary and self-affine, see \cite{obrien：torfs:vervaat:1990:article} Theorem 6.1. Therefore, it suffices to reduce  weak convergence  \eqref{eq:weak-convergence-empirical-rsm}  from $\mathsf{SM}([0,\infty))$ to $\mathsf{SM}([0,1])$ . 

\vspace{1em}

We now restrict the domain of $\mathcal{M}_\beta(\cdot)$ in \eqref{eq:limit-rsm-on-half-line} to $\mathcal{B}([0,1])$ and still denote it by $\mathcal{M}_\beta$ for convenience. A more intuitive Poisson point representation will appear. With repetitive use of notations, let us enumerate the point processes $\mathscr{P}$  in  $\reals \times [0,1] \times \mathcal{F}([0,1])$. 
\begin{itemize}
    \item let $(\Gamma_j : j \in \bbn)$ be arrival times of a unit rate Poisson process;
    \item let $(Q_j : j \in \bbn)$ be i.i.d.\ random variables with distribution 
    $$\PP(Q_1 \leq q ) = q^{1-\beta}, \quad q \in [0,1];$$
    \item let $(Z_j : j \in \bbn)$ be i.i.d.\ copies of the stable regenerative set in \eqref{eq:stable-regenerative-set-positive-real-axis}, and define i.i.d.\ random closed sets 
    \begin{equation}
        \label{eq:stationary-srs-unit-interval}
         R_j = (Q_j +Z_j ) \cap [0,1], \quad \forall \, j \in \bbn. 
    \end{equation}
 \end{itemize}   
Assume that the three families $(\Gamma_j)_{j \in \bbn}$,$(Q_j)_{j \in \bbn}$ and  $(Z_j)_{j \in \bbn}$ are independent, then we can rewrite the sup-derivative as
\begin{equation}
\label{eq:rsm-restricted-sup-derivative-eta}
    \eta (t) = 
\begin{cases}
    \sum_{j=1}^\infty -\log \Gamma_j  \one(t\in R_j), \quad & \text{if }  \sum_{j=1}^\infty \one(t\in R_j) = \ell_\beta \\
    -\infty , &\text{otherwise}
\end{cases}, \quad t \in [0,1].
\end{equation}
Then restricted random sup-measure has the representation 
\begin{equation}
    \label{eq:limit-rsm-restriction-unit-interval}
    \mathcal{M}_\beta (B) = \sup_{t\in B} \eta (t), \quad B \in \mathcal{B}([0,1]).
\end{equation}
After the restriction, we obtain a different form of the
the dichotomy \eqref{eq:zero-one-dichotomy-srs-intersections}.

\vspace{1em}

\begin{lemma}[Corollary B.3 in \cite{samorodnitsky:wang:2019:article}]
\label{lem:zero-positive-dicotomy}
    \phantom{blank}

    Set $\beta_s = s\beta - s+1$, then
\begin{equation}
    \label{eq:zero-one-dichotomy-srs-intersections-unit-interval}
   \PP\left( \bigcap_{j=1}^{s} R_{j}   \neq \emptyset \right) = 
    \begin{cases}
     \big[\Gamma(\beta) \Gamma(2-\beta) \big]^s \Big/  \Gamma(\beta_s) \Gamma(2-\beta_s)>0,  \quad  & s \leq \ell_\beta \\
     0,     & s > \ell_\beta 
    \end{cases}.
\end{equation}
Furthermore, for $s \leq \ell_\beta$  
\begin{equation}
    \label{eq:[0,1]-srs-intersections-first-time}
 \PP\left( \min \bigcap_{j=1}^{s} R_{j} \leq q \, \bigg| \,   
 \bigcap_{j=1}^{s} R_{j} \neq \emptyset \right) = q^{1-\beta_s}, \quad q\in [0,1].
\end{equation}
\end{lemma}

\vspace{1em}

\paragraph{Comparisons with Previous Limit Random Sup-Measures}
          
In the paper \cite{samorodnitsky:wang:2019:article}, the limit random sup-measure is 
    \begin{equation}
        \label{eq:Frechet-Limit-RSM-SW2019AoP}
        \mathcal{M}_{\alpha,\beta} (B) = \sup_{t\in B} \sum_{j=1}^\infty \Gamma_j^{1/\alpha} \one(t \in R_j),  \quad B \in \mathcal{G}([0,1])
    \end{equation}
    for all $ \alpha>0$ and $\beta \in (0,1)$. The supremum of the summation in \eqref{eq:Frechet-Limit-RSM-SW2019AoP} must contain $m$ terms because each  $\Gamma_j^{1/\alpha} $ is positive. 
However, it is wrong to write \eqref{eq:limit-rsm-restriction-unit-interval} as 
\begin{equation}
    \label{eq:wrong-limit-rsm-restriction-unit-interval}
     \mathcal{M}^\prime_\beta (B) = \sup_{t\in B} \sum_{j=1}^\infty  - \log \Gamma_j \one(t \in R_j).
\end{equation}
The random sup-measure $\mathcal{M}^\prime_\beta$ is a new one if  $\beta > 1/2$. For each $B$ and each $j_0$ such that $R_{j_0} \cap B \neq \emptyset$, there are a.s.\ infinitely many points $t\in B$ that uniquely belongs to $R_{j_0}$ because a stable regenerative set does not hit points.  Thus, on the event $\{ \Gamma_1 > 1 \}$,  all $-\log \Gamma_j < 0$ and hence
$$
\mathcal{M}^\prime_\beta (B) =  \bigvee _{j\geq 1} - \log \Gamma_j \one( R_j \cap B \neq \emptyset) \neq \mathcal{M}_\beta (B). 
$$
This explains the necessity of $\sum_{j=1}^\infty \one (t\in R_j) = \ell_\beta$ in \eqref{eq:limit-rsm-sup-derivative-eta} and \eqref{eq:rsm-restricted-sup-derivative-eta}.

\vspace{1em}

 However, once we take $\beta \in (0,1/2)$, then $R_j, j\in \bbn$ are a.s.\ mutually disjoint by \eqref{eq:zero-one-dichotomy-srs-intersections}. Correspondingly,         
        \eqref{eq:limit-rsm-restriction-unit-interval} can be simplified as
        $$
         \mathcal{M}_\beta =  \bigvee_{j\geq 1} - \log \Gamma_j  \one(R_j \cap B \neq \emptyset), \quad B \in \mathcal{G}([0,1
         ]),
        $$
        which is the limit random sup-measure in \cite{chen:samorodnitsky:2022:article}.

\paragraph{Non-Gumbel Distributions}

The distribution $H_\beta$ of the random variable $M_\beta \triangleq \mathcal{M}_\beta([0,1])$ is not of the Gumbel type for $\beta > 1/2$. We prove it by a counterargument.   If $H_\beta$ is of the Gumbel-type, i.e.\ for some $c>0$,  
$$
H_\beta (x) = \exp \left(- c e^{-x} \right),
$$
then $e^{M_\beta}$ will have a $1$-Fr\'echet distribution. However, this is a contraction because by \eqref{eq:rsm-restricted-sup-derivative-eta} and     \eqref{eq:limit-rsm-restriction-unit-interval}, 
$$
e^{M_\beta} \eid \mathcal{M}_{1,\beta}([0,1]) = \sup_{t\in[0,1]} \sum_{j=1}^\infty \Gamma_j^{1/\alpha} \one(t \in R_j)
$$
for the $\mathcal{M}_{1,\beta}$ in \eqref{eq:Frechet-Limit-RSM-SW2019AoP}. It is clearly that $\mathcal{M}_{1,\beta}$ is strictly non-Fr\'echet due to the overlapping of stable regenerative sets, see \cite{samorodnitsky:wang:2019:article} \S 3.

\subsection{Extremal Processes}
\label{subsec:extremal-processes}

 Let $(\mathcal{E}_0(t): t >0)$ denote the standard  extremal process with Gumbel margins, i.e.\ $\mathcal{E}_0(1) \eid \Lambda$. We recall from (4.19) and Proposition 4.7 $(\rom1)$-$(\rom3)$ in \cite{resnick:1987:book} some basic facts. 
 \begin{itemize}
     \item   It is non-decreasing and 
      $$
     \lim_{t\to 0} \mathcal{E}_0(t) \xlongequal{a.s.} -\infty, \quad \lim_{t\to \infty} \mathcal{E}_0(t) \xlongequal{a.s.} \infty.
     $$
    It is  continuous in probability
     and admits a version in $D(0,\infty)=\bigcap_{\epsilon>0} D[\epsilon,\infty)$;
     \item  For all $0<t_1<\cdots<t_k < \infty$ and $-\infty< x_1\leq \cdots \leq x_k < \infty$, the finitely dimensional distribution is 
     \begin{equation*}
         \PP\big(\mathcal{E}_0(t_i) \leq x_i, i=1,\ldots, k  \big) = \exp\left( -\sum_{i=1}^k (t_i - t_{i-1}) e^{-x_i}  \right).
     \end{equation*}
 \end{itemize}

 \vspace{1em}
 
The random sup-measure $\mathcal{M}_\beta(\cdot)$ in \eqref{eq:limit-rsm-on-half-line} induces another extremal process 
\begin{equation}
    \label{eq:limit-extremal-process}
    \big(\mathcal{E}_\beta(t): t >0 \big), \quad \mathcal{E}_\beta(t) \triangleq \mathcal{M}_\beta([0,t]).
\end{equation}
If  $\beta\in (0,1/2]$, then  we claim
\begin{equation}
\label{eq:extremal-process-time-change-formula}
    (\mathcal{E}_\beta(t): t >0)  \eid (\mathcal{E}_0(t^{1-\beta}): t >0).
\end{equation}
This slightly extends the statement (4.4) in \cite{chen:samorodnitsky:2022:article}, where $\beta \in (0,1/2)$ ibid and the proofs are the same. 
However, in the case $\beta\in (1/2,1)$, the  time-change formula \eqref{eq:extremal-process-time-change-formula} will strictly fail. We derive the following comparison between finitely dimensional distributions. 

\vspace{1em}

\begin{proposition}[A Stochastic Dominance Relation]
    \label{proposition:time-change-formula-break-down}
    \phantom{blank}

    If $\beta\in (1/2,1)$, then  for all $0<t_1<\cdots<t_k < \infty$ and $-\infty< x_1\leq \cdots \leq x_k < \infty$,
     \begin{equation*}
           \PP\big(\mathcal{E}_\beta(t_i) \leq x_i, i=1,\ldots, k  \big) >
           \PP\big(\mathcal{E}_0(t_i^{1-\beta}) \leq x_i, i=1,\ldots, k  \big).
     \end{equation*}
    
    \end{proposition}

\begin{proof}[Proof of Proposition \ref{proposition:time-change-formula-break-down}]
    \phantom{blank}
    
For both processes $\mathcal{E}_0$ and  $\mathcal{E}_\beta$, they satisfy the same self-affine property 
$$
(\mathcal{E}(at): t>0) \eid (\mathcal{E}(t) + (1-\beta) \log a: t>0).
$$ 
It follows that for any choice of $0<t_1<\cdots<t_k < \infty$ and $-\infty< x_1\leq \cdots \leq x_k < \infty$, 
$$
\PP\big(\mathcal{E}(t_i) \leq x_i, i=1,\ldots, k  \big) = 
\PP\big(\mathcal{E}(at_i) \leq x_i + (1-\beta)\log a, i=1,\ldots, k  \big). 
$$
We can therefore choose $a>0$ small enough such that all $a t_i < 1$ and all $ x_i + (1-\beta)\log a < 0$. We thus assume that
 $0<t_1<\ldots < t_k < 1$ and $-\infty< x_1\leq \cdots \leq x_k < 0$ without loss of generality. 
Let us further take $k=2$, which highlight the main spirit. 
In the representations \eqref{eq:rsm-restricted-sup-derivative-eta} and \eqref{eq:limit-rsm-restriction-unit-interval}, 
the sequence  $-\log \Gamma_j $ is strictly decreasing. Combine this fact with  $x_1 \leq x_2 <0$, we obtain that
\begin{align*}
     & \PP ( \mathcal{E}_\beta (t_1) < x_1,  \mathcal{E}_\beta (t_2) < x_2 ) \\
    > & \mathbb{P} 
    \Big(
    \max \{ -\log \Gamma_j : R_j \cap (0,t_1) \neq \emptyset \} \leq x_1, \\
    & \phantom{\PP \Big) }  
    \max\{  -\log \Gamma_j : R_j \cap (0,t_1) =\emptyset,  R_j \cap (0,t_2) \neq \emptyset \} \leq x_2 \Big) \\
    = & \mathbb{P} 
    \Big(
    \max \big\{  -\log \Gamma_j :Q_j \in (0,t_1)    \big\} \leq x_1 \Big) \cdot \mathbb{P} 
    \Big(
    \max \big\{  -\log \Gamma_j : Q_j \in (t_1,t_2)    \big\} \leq x_2 \Big) \\
    = &  \PP \left( \mathcal{E}_0 (t_1^{1-\beta}) < x_1,  \mathcal{E}_0 (t_2^{1-\beta}) < x_2 \right ).
\end{align*}
    
\end{proof}


\section{Dynamics of Null-Recurrent Markov Chains}
\label{sec:DNRMC}

\subsection{Construction}
\label{subsec:DNRMC-construction}

Let $(Y(t): t\in \bbz)$ be an \textit{irreducible, aperiodic} and \textit{null-recurrent} Markov chain on $\bbz$. Denote by $(\pi_i:i\in \bbz)$ the unique invariant measure by taking $\pi_0=1$. On the \textit{path space} $\mathbb{Y} = \mathbb{Z}^\mathbb{Z}$, there exist naturally
 the \textit{cylindrical $\sigma$-field} $\mathcal{Y}=\sigma(\mathbb{Z}^\mathbb{Z})$ and
the \textit{(left) shift operator}
\begin{equation}
    \label{eq:def:left-shift-operator}
    \tau: (\ldots, y_0, y_1 ,\ldots ) \mapsto (\ldots, y_1, y_2,\ldots).
\end{equation}
Equipped with the  measure
\begin{equation}
\mu(B) = \sum_{i\in \bbz} \pi_i \mathbb{P}(B \, | \, Y_0 = i) , \quad \forall \, B \in \mathcal{Y},
\end{equation} 
then the dynamical system $(\mathbb{Y},\mathcal{Y},\mu, \tau)$  is \textit{measure preserving}, \textit{conservative} and \textit{ergodic};
see \cite{harris:robbins:1953:article}.
By conditioning on the event $\{Y(0)=0\}$,  successive returns to $0$ form a renewal process which characterizes
quantitatively  the long
memory built in $(\mathbb{Y},\mathcal{Y},\mu, \tau)$.

\vspace{1em}

\begin{assumption}[Memory Parameter and Return Probabilities] 
\label{assumption:memory-parameter-and-return-probabilities}
\phantom{blank}

Denote by $\varphi:\mathbb{Y} \to \mathbb{N}$ the first return time to $0$,
\begin{equation}
    \label{eq:MC-dynamics-return-times}
    \varphi(y) = \min \{  t \geq 1: y_t = 0  \} , \quad \forall\, y = (y_t)_{t\in \bbz} \in \mathbb{Y};
\end{equation}
and denote by $F_\beta$ the (conditional) distribution of the renewal epoch, 
\begin{equation}
\label{eq:MC-dynamics-conditional-renewal-epoch-distribution}
    F_\beta(n) = \PP( \varphi \leq n \, | \,  Y_0 = 0 ) , \quad \forall\, n \in \bbn_0.
\end{equation}

\begin{description}
    \item[$(\rom1)$] Assume that the memory parameter 
\begin{equation}
    \label{eq:memory-parameter-full-range-unit-interval}
    \beta\in (0,1).
\end{equation}
    \item[$(\rom2)$] Assume that there is a 
    slowly varying function $L_\beta(\cdot)$ such that   
\begin{align}
    & \overline{F_\beta}(n)  = n^{-\beta} L_\beta (n),  \label{eq:assumption-regularly-varying-return-times}
    \\
    &
    \sup_{n\geq 0} \frac{ n  \PP( \varphi = n \, | \, Y_0 =0)}{ \overline{F_\beta} (n) } < \infty .
     \label{eq:assumption-doney-condition}
\end{align}
\end{description}

\end{assumption}

\vspace{1em}

\begin{remark}
    \label{rmk:remark-of-range-of-memory-parameter}
    { \rm
    \phantom{blank}

    \begin{itemize}
        \item In the construction of the Markov chain $(Y(t): t\in \bbz)$, one can replace the state space $\bbz$ by an infinitely countable set $\mathbb{S}$. And, correspondingly, replace the state $0$ by any $s_0\in \mathbb{S}$. See, for example, \cite{lacaux:samorodnitsky:2014:article} and
        \cite{samorodnitsky:wang:2019:article}. This replacement does not the main spirit of all proofs.

        \item The assumption  \eqref{eq:assumption-doney-condition} is indispensable in renewal theoretical statements, see \cite{doney:1997:article} Theorem B.
    \end{itemize}
    }
\end{remark}

\subsection{Intersections of Return Times}
\label{subsec:intersections-of-return-times}

The long time behavior of the return times plays an extremely important role in the sequel. Due to 
the null-recurrence, we can not  replace the infinite measure $\mu$ by an invariant probability measure. We shall instead construct a sequence of probability measures $(\mu_n: n \in \bbn_0)$ and derive correspondingly a weak convergence theorem, see Theorem \ref{thm:key-preliminary-theorem-far-away-intersection}. Most proofs in this subsection require a systematic preparation from
renewal theory. We thus include the supplements in and defer proofs to \S \ref{subsec:supplement-renewal-theory-and-proofs}.

\vspace{1em}

Let us fix some integer $m\geq 2$. Henceforth in   \S \ref{subsec:intersections-of-return-times}, we  assume and write  respectively  that  
\begin{equation}
    \label{eq:local-range-of-beta}
    \beta\in \left(  \frac{m-1}{m}, \frac{m}{m+1}  \right), \quad  \betast \triangleq m\beta - (m-1).
\end{equation}
It is easy to note that $0< \betast < 1-\beta < 1/2$.
The 
\textit{wandering rate sequence} $(w_n: n\in \bbn_0)$ is given by 
\begin{equation}
    \label{eq:definition-wandering-rate}
    w_n = \mu \left(  \cup_{k=0}^n A_k  \right)  \quad \text{with}
    \quad A_k \triangleq \tau^{-k}(A_0) =  \{ y = (y_t)_{t\in \bbz} \in \mathbb{Y}: y_k = 0   \}.
\end{equation}
We recall from Remark 2.1 in \cite{chen:samorodnitsky:2022:article} that $w_n \in \mathsf{RV}_{1-\beta}$ and asymptotically
\begin{equation}
\label{eq:wandering-rate-asymptotic}
    w_n \sim \frac{n^{1-\beta} L_\beta(n)}{1-\beta}.
\end{equation}
For each $n$, the quantity $w_n$ scales a probability measure $(\mathbb{Y}, \mu_n(\cdot)$ on $\mathcal{Y})$ as 
\begin{equation}
\label{eq:probab-measure-mu-n}
    \mu_n (\cdot) =  \mu  \big( \cdot \cap \,  (\cup_{k=0}^n A_k) \big) \, \Big/ \, w_n.   
\end{equation}

\vspace{1em}

Let $ \{ Y_{j;n}(t): t\in \bbz \}, j=1,2,3,\ldots$ be i.i.d.\ random elements with the law $\mu_n$. For each $j$ and $n$, we denote the (non-empty) set of return times of  $Y_{j;n}$ by
\begin{equation}
    \label{eq:return-times-random-zeros-Ikn}
    I_{j,n} = \{ 0 \leq t \leq n: Y_{j,n} (t) = 0  \} .
\end{equation}
Due to the range \eqref{eq:local-range-of-beta} of $\beta$, the central object of this subsection is the intersection $\bigcap_{j=1}^m I_{j,n}$. Its various properties are required in the sequel.
  
\vspace{1em}

\begin{lemma}[Theorem 5.4 and Corollary B.3 in \cite{samorodnitsky:wang:2019:article}]
    \label{lem:intersections-from-samorodnitsky-wang-2019AoP}
    \phantom{blank}

    \begin{itemize}
        \item[$(\rom1)$]  Let $R_j$ be in \eqref{eq:stationary-srs-unit-interval}. Then,
        for any integer $N \in \bbn$,  
    \begin{equation}
    \label{eq:convergence-of-intersections-samorodnitsky&wang;1}
        \left(  \frac{1}{n} \bigcap_{j\in J} I_{j,n}   \right)_{  J \subset \{1,\ldots, N \}   } 
        \Rightarrow   \left(  \bigcap_{j\in \mathscr{J}} R_j   \right)_{  \mathscr{J} \subset \{ 1,\ldots, N \}   }.
    \end{equation} 
      
        \item[$(\rom2)$] Let $Q_\betast$ be a random variable with distribution 
      \begin{equation}
          \label{eq:distribution-V-betast}
          \PP( Q_\betast\leq q ) = q ^{1-\betast} \in [0,1].
      \end{equation}     
      Then, we have the the conditional weak convergence
      \begin{equation}
          \label{eq:convergence-of-first-simultaneous-intersection-time}
           \frac{\min \bigcap_{i=1}^m  I_{i,n} }{n} \xLongrightarrow{ \PP \left( \cdot \, \big| \, \bigcap_{i=1}^m  I_{i,n} \neq \emptyset \right)}
       Q_\betast.
      \end{equation}
      Hence, by  \eqref{eq:zero-one-dichotomy-srs-intersections-unit-interval} and   \eqref{eq:weak-convergence-of-random-closed-set},
       \begin{equation}
          \label{eq:probability-non-empty-m-intersections}
          \lim_{n\to \infty}  \PP \left(  \bigcap_{i=1}^m  I_{i,n} \neq \emptyset \right)= \frac{\big[ \Gamma(\beta) \Gamma(1-\beta) \big]^m}{  \Gamma(\betast) \Gamma(2-\betast) }  \in (0,1).
      \end{equation}

    \end{itemize}
\end{lemma}

\vspace{1em}

\begin{proposition}
\label{prop:more-intersection-properties}
  \phantom{blank}
  
    \begin{itemize}
        \item[$(\rom1)$] For any open interval $B \subset [0,1]$, 
        \begin{equation}
       \label{eq:finite-intersection-properties-non-empty} 
       \lim_{C \to \infty} \lim_{n\to \infty}
       \PP\left(  \exists \, 1\leq j_{1} < \cdots <  j_m \leq C  \text{ such that } nB \cap  \bigcap_{i=1}^m I_{j_i,n}  \neq \emptyset \right) = 1.
    \end{equation}
        \item[$(\rom2)$]  Take any $\delta \in (0, 1-\beta - (m+1)^{-1})$, then
        \begin{equation}
       \label{eq:finite-intersection-properties-empty} 
       \lim_{n\to \infty} \PP\left( \exists \, 1\leq 
       j_1 < \cdots < j_{m+1} \leq n^\delta \text{ such that }\bigcap_{i=1}^{m+1} I_{j_i,n}  \neq \emptyset  \right) = 0.
    \end{equation}

    \item[$(\rom3)$] 
    Let $(\vartheta_n)$ be the sequence \eqref{eq:sequence-vartheta-n}. Then 
    for any $C>0$, there exists $c>0$ such that 
\begin{equation}
         \label{eq;m-intersections-cardinality-control;log-power-law-tail}
      \PP \left( \left| \bigcap_{i=1}^m  I_{i,n} \right| \geq  c\vartheta_n \log n
       \right) \leq n^{-C}.
     \end{equation}
    \item[$(\rom4)$]   Let $\{ \mathscr{Z}_\betast (t):t \in \bbr_+  \}$ be  the Mittag-Leffler process \eqref{eq:Mittag-Leffler-process} with self-similar index $\betast$ and $V_\betast$ be in \eqref{eq:distribution-V-betast}. Assume that $\{ \mathscr{Z}_\betast (t):t \in \bbr_+  \}$ and $Q_\betast$ are independent, then for any $b\in (0,1)$
     \begin{equation}
         \label{eq;m-intersections-cardinality-weak-convergence}
       \frac{ \left| (0,nb) \cap  \bigcap_{i=1}^m  I_{i,n} \right| } { \vartheta_n } \xLongrightarrow{ \PP \left( \cdot \, \big| \, (0,nb) \cap \bigcap_{i=1}^m  I_{i,n} \neq \emptyset \right)}
        b^\betast \mathscr{Z}_\betast ( 1 - Q_\betast ).
     \end{equation}
\end{itemize}

\end{proposition}

\vspace{1em}

We now formulate the last result in this section. 
For each open interval $B \subset [0,1]$, 
denote by 
$\Omega_{n,B}$ the event $ \left\{  nB \cap \bigcap_{j=1}^m I_{j,n} \neq \emptyset   \right  \}$, $\mathscr{F}_{n,B}$ the restricted $\sigma$-field  $\sigma(I_{j,n}, j=1,\ldots, m) \big|_{ \Omega_{n,B}}$, and $\boldsymbol{P}_{n,B}$ the conditional probability measure $\PP\left(\cdot \, | \,  \mathscr{F}_{n,B} \right)$. There are two related random variables of our interests,
\begin{align}
  & p_n (B) = \PP\left( I_{m+1,n} \cap nB \cap \bigcap_{j=1}^m I_{j,n} \neq \emptyset \, \bigg| \, \mathscr{F}_{n,B}  \right), 
    \label{eq:random-intersection-probability-conditional-probability}\\
  &  \ell_n(B)  = 
     \min \left\{  j\geq m+1:  I_{j,n} \cap  nB \cap \bigcap_{i=1}^m I_{i,n}  \neq \emptyset \right\}.
    \label{eq:definition-first-intersection-after-m} 
\end{align}
Obviously,  $\ell_n (B)- m$ is geometrically distributed with success probability   $ p_n (B)$.

\vspace{1em}

\begin{theorem} 
\label{thm:key-preliminary-theorem-far-away-intersection}
\phantom{blank}

Let $(w_n: n\in \bbn)$ be the sequence \eqref{eq:definition-wandering-rate}, $(\vartheta_n: n\in \bbn)$ be the sequence \eqref{eq:sequence-vartheta-n}  and  $p_\mathsf{escape}$ be the constant \eqref{eq:escape-probability-(m+1)-intersections}.
Let $\{ \mathscr{Z}_\betast (t):t \in \bbr_+  \}$ be  the Mittag-Leffler process \eqref{eq:Mittag-Leffler-process} with self-similar index $\betast$, $Q_\betast$ be the random variable \eqref{eq:distribution-V-betast} and $\Gamma$ be a standard exponential random variable. Assume further that $\{ \mathscr{Z}_\betast (t):t \in \bbr_+  \}, Q_\betast$ and $\Gamma$ are independent. Then, for any open interval
$B\subset [0,1]$ with length  $d_B$, 
\begin{equation}
    \label{eq:key-preliminary-joint-far-away-intersection}
      \left( \frac{ w_n p_n (B) }{ \vartheta_n}, p_n (B) \ell_{n} (B) \right) 
   \xLongrightarrow{ \boldsymbol{P}_{n,B} }
    \left( p_\mathsf{escape} d_B^\betast \mathscr{Z}_\betast (1-Q_\betast) , \Gamma \right) 
\end{equation}
\end{theorem}

\subsection{Supplements and Proofs}  
\label{subsec:supplement-renewal-theory-and-proofs}

\textbf{Convention:} Unless we point out the initial position $S(0)$, a renewal process $(S(t):t\in \bbn_0)$ will always start from $S(0)=0$.

\subsubsection{Individual Renewal}
\label{subsubsec:individiaul-renewal}

Let us assume $\beta\in (0,1)$ in \S \ref{subsubsec:individiaul-renewal}.  Let $\boldsymbol{S}_\beta=(S_\beta (t): t \in \bbn_0)$ be a renewal process with tail epoch distribution $\overline{F_\beta}(n)  = n^{-\beta} L_\beta (n)$ in 
 Assumption \ref{assumption:memory-parameter-and-return-probabilities}.
Standard results shows that 
 \begin{align*}
     & u_\beta (n) = \PP(\boldsymbol{S}_\beta \text{ renews at time } n) \sim \frac{n^{\beta-1}}{\Gamma(\beta)\Gamma(1-\beta) L_\beta(n)}, 
     \\
     & U_\beta (n) = \sum_{t=0}^n u_\beta (n) \sim \frac{n^\beta}{\Gamma(1+\beta)\Gamma(1-\beta)L_\beta(n)} 
     \asymp \frac{1}{\overline{F_\beta} (n)}.
 \end{align*}

\vspace{1em}

\begin{lemma}
\label{lem:cardinality-estimate-range-of-S-beta}
\phantom{blank}

   Denote by $\mathsf{Range}_\beta = \{ S_\beta(0), S_\beta(1), S_\beta(2), \ldots \}$ the full range of $\boldsymbol{S}_\beta$. Then,
\begin{itemize}
    \item[$(\rom1)$] $\EE  \left|  
    \mathsf{Range}_\beta \cap \{0,\ldots,n\}
    \right|   =  U_\beta (n)$.
    \item[$(\rom2)$] For any $C>0$, there exists $c>0$ such that, for $n \gg 0$,
    \begin{equation}
    \label{eq:upper-bounds-on-cardinality-single-range}
   \PP \left( \left|  
    \mathsf{Range}_\beta \cap \{0,\ldots,n\}
    \right|  \geq 
    c  U_\beta (n) \log n
    \right)  \leq n^{-C}.
\end{equation}
\end{itemize}
\end{lemma}

\begin{proof}[Proof of Lemma \ref{lem:cardinality-estimate-range-of-S-beta}]
    \phantom{blank}

    \textbf{Part (i).} Note that 
    \begin{align*}
      \left|  \mathsf{Range}_\beta \cap \{0,\ldots,n\} \right| = \sum_{i=0}^n \mathbf{1}(\boldsymbol{S}_\beta \text{ renews at } i),
    \end{align*}
    and hence taking expectations on both sides gives the desired equality.

    \vspace{1em}

    \textbf{Part (ii).} For any finite set $W \subset \bbz$ and $k\in \bbn$, we define its \textit{sojourn time} as 
    $$
    T_W (k) = | \{ 0 \leq j \leq k: S_\beta (j) \in W \} |.
    $$
    Denote by $D(W)=\{ x-y: x , y \in W  \}$ the \textit{difference set}. Then, for every $\delta \in (0,1)$, there exists $\lambda_0=\lambda_0(\delta)$ that does not rely on $W$, such that whenever $\lambda > \lambda_0$ and $0<\EE T_{D(W)} (k)< \infty$, we have 
    $$
    \PP  \left(   T_B(k) \geq \lambda  \EE T_{D(W)} (k) \right)  \ \leq e^{-\lambda \delta},
    $$
    see Lemma 3.1 in \cite{pruitt:taylor:1969:article}.  Let $W=\{0,\ldots,n\}$ for our purpose.  We obviously have $T_W=T_{D(W)} =  \mathsf{Range}_\beta \cap \{0,\ldots,n\}$ because $\boldsymbol{S}_\beta$ is strictly increasing. And we have $0<\EE T_{D(W)} (k)< \infty$ by part (i). Thus,  \eqref{eq:upper-bounds-on-cardinality-single-range} follows by taking $\lambda = c \log n$ for some sufficiently large $c$ so that $e^{-\lambda \delta} \leq n^{-c}$.
\end{proof}

\subsubsection{Simutaneous Renewal}
\label{subsubsec:simultaneous-renewal}

For $i=1,\ldots, m+1$, let $\boldsymbol{S}_{\beta,i}=(S_{\beta,i}(t): t\in \bbn_0 )$ be independent copies of the $\boldsymbol{S}_\beta = (S_\beta (t): t \in \bbn_0)$ above with ranges denoted by
\begin{equation}
    \label{eq:range-of-(m)-renewal-processes}
    \mathsf{Range}_{\beta,i}  = \{ S_{\beta,i}(0),  S_{\beta,i}(1),  S_{\beta,i}(2), \ldots   \}.
\end{equation}
\textbf{In addition, we shall now assume that \eqref{eq:local-range-of-beta} holds until the end of \S \ref{sec:DNRMC}.}

\vspace{1em}

Consider first the  simultaneous renewal times
\begin{equation}
   \bigcap_{i=1}^m    \mathsf{Range}_{\beta,i}
    \label{eq:simultaneous-renewal-times-m-renewal-process;1}.
\end{equation}
for all $m$ individual renewal processes $\boldsymbol{S}_{\beta,i}, i=1,\ldots, m$. 
From Lemma A.1 in \cite{samorodnitsky:wang:2019:article}, the simultaneous renewal epoch has the tail distribution  
\begin{equation}
    \label{eq:m-simultaneous-epoch-distribution;1}
    \overline{F_{\beta_\ast}} (n)= n^{-\beta_\ast} L_{\beta_\ast}(n).
\end{equation}
with the $\betast$ in \eqref{eq:local-range-of-beta} and 
 \begin{equation}
      \label{eq:m-simultaneous-epoch-distribution;2}
      \quad  L_{\betast} (n) \sim \frac{\Big(\Gamma(\beta) \Gamma(1-\beta) L_\beta(n) \Big)^m}{\Gamma(\betast) \Gamma(1-\betast)}.
 \end{equation}

\vspace{1em}

We shall also need the  simultaneous renewal times
\begin{equation}
    \bigcap_{i=1}^{m+1}   \mathsf{Range}_{\beta,i} 
    \label{eq:intersection-(m+1)-renewal-process-range-containing-zero;1},
\end{equation}
for all $\boldsymbol{S}_{\beta,i}, i=1,\ldots, m+1$. The intersection 
 \eqref{eq:intersection-(m+1)-renewal-process-range-containing-zero;1}
corresponds to the range of a terminal renewal process. The key is to  check that
$
\sum_{n=1}^\infty \big(u_\beta (n) \big)^{m+1} < \infty 
$
under condition 
\eqref{eq:local-range-of-beta}. We leave it to the reader. Therefore we derive the existence of  following 
 probability 
\begin{equation}
    \label{eq:escape-probability-(m+1)-intersections}
    p_\mathsf{escape} =
    \PP \left(  
    \bigcap_{i=1}^{m+1} \mathsf{Range}_{\beta,i}  = \{0 \}  
    \right  ) \in (0,1).
\end{equation}


\subsubsection{Convergence of Cardinalities}
\label{subsubsec:convergence-cardinality}

With respect to \eqref{eq:m-simultaneous-epoch-distribution;2}, let  $(c_{n}, n\in \bbn)$ be the sequence such that 
\begin{equation}
\label{eq:constant-stable-distribution-cn}
   \lim_{n\to \infty} n \overline{F_\betast}(c_n) = C_{\betast} \triangleq
\left(  \int_0^\infty x^{-\betast} \sin x dx \right) ^ {-1} = \frac{1-\betast}{\Gamma(2-\betast)\cos(\pi \betast / 2)}. 
\end{equation}
For convenience, we shall denote  by 
\begin{align}
    &\boldsymbol{S}_\betast = \big(S_\betast (t): t\in \bbn_0 \big), \quad 
    \mathsf{Range}_\betast = \{  S_\betast (0), S_\betast (1), \ldots \}; 
     \label{eq:simultaneous-renewal-times-m-renewal-process;2} \\
    &\mathsf{Range}_\betast = \bigcap_{i=1}^m  \mathsf{Range}(\boldsymbol{S}_{\beta,i}).
     \label{eq:simultaneous-renewal-times-m-renewal-process;3} 
\end{align}
the $m$-simultaneous renewal process and its range.
The goal is to establish the weak convergence of the cardinality 
$| \mathsf{Range}_{\betast} \cap \{0,\ldots,n \}  |$ for $n\to \infty$. We shall first recall the 
weak convergence of $\boldsymbol{S}_\betast$ to the $\betast$-stable subordinator.

\vspace{1em}

\begin{lemma}[Theorems 4.5.1 and 4.5.3 in \cite{whitt:2002:book}]
\label{prop:convergence-to-stable-subordinator}
\phantom{blank}

Let
$\{ \mathscr{Y}_\betast (t) , t \in \reals_+ \}$ be the $\betast$-stable subordinator   in \S \ref{subsubsec:subordinator-and-regenerative-sets}, then
\begin{equation}
    \label{eq:simultaneous-returns-convergence-to-stable-levy-motion}
    \left \{ \frac{ S_{\betast} ( \lfloor nt \rfloor ) }{ c_n }  : t \in \bbr_+
\right \}
\xLongrightarrow{(D(\bbr_+), J_1)} 
\{ \mathscr{Y}_{\betast}(t): t \in \bbr_+ \}.
\end{equation}

\end{lemma}

\vspace{1em}

  We next note that, for any integer $k\in \bbn_0$, there is the identity  
  \begin{equation}
      \label{eq:range-and-inversed-renewal-identity}
      \Big|  \mathsf{Range}_\betast \cap \{0,1,\ldots, k   \}  \Big|  = (S_\betast )^\leftarrow (k).
  \end{equation}
Hence, to prove the weak convergence of cardinalities, it suffices to take inverses on both sides of \eqref{eq:simultaneous-returns-convergence-to-stable-levy-motion}. For this reversed convergence, we define a new scaling sequence $(\vartheta_n: n \in \bbn)$ as
\begin{equation}
    \label{eq:sequence-vartheta-n}
    \vartheta_n =  C_{\betast}  \frac{n^\betast}{L_\betast(n)}.
\end{equation}
It is tedious but straightforward to verify the asymptotic relation 
\begin{equation}
    \label{eq:vatheta-n-and-c-n-are-inverse-of-each-other}
    c_{\lfloor \vartheta_n \rfloor} \sim \vartheta_{ \lfloor c_n  \rfloor } \sim n.
\end{equation}

\vspace{1em}

\begin{proposition} 
\label{prop:convergence-to-MittagLeffler-process-J1-topology}
\phantom{blank}

Let $\{ \mathscr{Z}_\betast(t):t\in \bbr_+  \}$ be the Mittag-Leffler process \eqref{eq:Mittag-Leffler-process} of index $\betast$, then
\begin{equation}
     \label{eq:inverse-convergence-to-MittagLeffler-process-J1-Skorokhod}
     \left\{
        \frac{ (S_\betast)^\leftarrow (\lfloor nt \rfloor )}{\vartheta_n}: t \in \reals_+
        \right\} 
        \xLongrightarrow{(D(\reals_+), J_1)}
        \left\{ \mathscr{Z}_\betast(t):t\in \bbr_+  \right\}.
\end{equation}
\end{proposition}

\begin{proof}[Proof of Proposition \ref{prop:convergence-to-MittagLeffler-process-J1-topology}]
     \phantom{blank}

     For convenience, we write 
     $G_n (t) =   S_\betast ( \lfloor nt \rfloor ) / c_n$.
     Combine Lemma \ref{prop:convergence-to-stable-subordinator} and the Theorem in \cite{whitt:1971:article}, we immediately obtain that
     \begin{equation*}
         \{ (G_n)^\leftarrow (t): t \in \reals_+  \}
        \xLongrightarrow{(D(\reals_+), M_1)}
        \left\{  \mathscr{Y}_{\betast} ^\leftarrow (t): t \in \reals_+  \right\} = 
        \left\{ \mathscr{Z}_\betast(t):t\in \bbr_+  \right\}.
     \end{equation*}
     Because $\left\{ \mathscr{Z}_\betast(t): t \in \reals_+  \right\}$ is a.s.\ continuous, convergences in $M_1$ and $J_1$ topologies coincide. We therefore conclude that
$$
         \{ (G_n)^\leftarrow (t): t \in \bbr_+  \}
        \xLongrightarrow{(D(\reals_+), J_1)}
        \left\{  \mathscr{Z}_\betast(t):  t \in \bbr_+  \right\}.
$$
A simple algebraic manipulation yields that $(G_n)^\leftarrow (t)=\frac{1}{n} (S_\betast)^\leftarrow (c_n t)$.
By replacing  $n$ by   $\lfloor \vartheta_n \rfloor$, we get
$$
(G_{ \lfloor \vartheta_n \rfloor } )^\leftarrow (t) = 
\frac{1}{ \lfloor \vartheta_n \rfloor} (S_\betast)^\leftarrow (c_{ \lfloor \vartheta_n \rfloor } t)
\sim \frac{1}{  \vartheta_n } (S_\betast)^\leftarrow (c_{ \lfloor \vartheta_n \rfloor } t).
$$
By referring to  the definition of the $J_1$ metric and \eqref{eq:vatheta-n-and-c-n-are-inverse-of-each-other}, 
  we thus obtain \eqref{eq:inverse-convergence-to-MittagLeffler-process-J1-Skorokhod}.
\end{proof}

\subsubsection{Convergence of Capacities}
\label{subsubsec:convergence-capacity}

For any transient random walk $\boldsymbol{S}=(S(t): t \in \bbn_0)$ on $\bbz$, we recall from  \cite{spitzer:1964:book} \S 25 \textbf{D3}  the \textit{capacity} functional
\begin{equation}
    \label{eq:capacity-of-a-set-A}
    \mathsf{Capacity}(A) = \sum_{a\in A} \PP(\boldsymbol{S} \text{ escapes from } A \, | \, S(0)=a   ), \quad \text{ for any finite } A \subset \bbz.
\end{equation}
We are interested in capacities of random sets.
 First, truncate $\mathsf{Range}_\betast$ into a triangular array as 
\begin{equation}
    \mathsf{Range}_\betast(n_1, n_2) = \{ S_{\betast}(n_1), \ldots, S_{\betast}(n_2-1)   \}, \quad \forall \, 0 \leq n_1 < n_2.
\end{equation}
 Second, we take $\boldsymbol{S}=\boldsymbol{S}_{\beta,m+1}$ which is independent of $\boldsymbol{S}_{\beta,1},\ldots, \boldsymbol{S}_{\beta,m}$. This yields random capacities
    \begin{align*}
     \mathsf{CR}_{n_1, n_2} 
    \triangleq &  \mathsf{Capacity} \Big( \mathsf{Range}_\betast ( n_1, n_2) \Big) \\
    = &  \sum_{a\in \mathsf{Range}_\betast ( n_1, n_2) } \PP \left(\boldsymbol{S} \text{ escapes from }  \mathsf{Range}_\betast(n_1, n_2) \, | \, S(0)=a  \right ).
\end{align*}

\vspace{1em}

\begin{proposition}\label{prop:convergence-of-range-capacities-subadditive-argument}
\phantom{blank}

    Let $p_\mathsf{escape}$ be the escape probability \eqref{eq:escape-probability-(m+1)-intersections}, then 
    \begin{equation}
        \label{eq:convergence-of-range-capacities-subadditive-argument}
        \frac{  \mathsf{CR}_{0, n} }{n} 
        \xlongrightarrow{a.s. \; \& \; L^1} p_\mathsf{escape}.
    \end{equation}
\end{proposition}

\begin{proof}[Proof of Proposition \ref{prop:convergence-of-range-capacities-subadditive-argument}]
    \phantom{blank}

    We note that the triangular array $\{ \mathsf{CR}_{n_1, n_2}, 0 \leq n_1 <n_2 \}$ satisfies the following  properties.   
    \begin{itemize}
        \item[$(\rom1)$] For any $n, k \in \bbn$, $\mathsf{CR}_{0,n}+\mathsf{CR}_{n,n+k} \geq \mathsf{CR}_{0,n+k}$. This follows from the fact that the capacity functional  is sub-additive; see \textbf{P11} in 
        \cite{spitzer:1964:book} \S 25.
        \item[$(\rom2)$] For any $k$, the sequence $\{ \mathsf{CR}_{nk, (n+1)k}, n \geq 1 \}$ is i.i.d., hence stationary and ergodic. 
        \item[$(\rom3)$] For any $k$, the sequence $\{ \mathsf{CR}_{k, n+k}, n \geq 1 \} \eid \{ \mathsf{CR}_{0, n}, n \geq 1 \}$.
        \item[$(\rom4)$] $\mathsf{CR}_{0, 1} =  (\mathsf{CR}_{0, 1})^+ \geq 0$, $\EE   \mathsf{CR}_{0, 1}  <1$ and $\EE ( \mathsf{CR}_{0, n} ) \geq 0 $. 
    \end{itemize}
    Properties $(\rom1)$ to $(\rom4)$ match up with the subadditive ergodic theorem, see \cite{durret:PTE5:2019} Theorem 4.6.1. We automatically obtain the a.s.\ and $L^1$ convergences to some non-negative real number, which is directly seen to be $p_\mathsf{escape}$.
\end{proof}

\subsubsection{Proofs in \S \ref{subsec:intersections-of-return-times}}
\label{subsubsec:proof-in-intersection-return-times}

Let us now back to the delayed proofs of \S \ref{subsec:intersections-of-return-times}. We first note that each random set $I_{j,n}$ is a truncated range of a delayed renewal process. More precisely, 
under the framework of \S \ref{subsubsec:simultaneous-renewal}, 
let us take renewal process $\boldsymbol{S}_{\beta,i},i=1,\ldots,m+1$ independent of $ I_{i,n}, i=1,\ldots,m+1$. Then,  due to the convention $\min \emptyset  = \infty$, we always have
\begin{align}
    & \left( I_{i,n}  \right)_{i=1,\ldots,m+1} \eid  
    \left( ( \min I_{i,n} + \mathsf{Range}_{\beta,i } ) \cap \{0,\ldots,n\} \right)_{i=1,\ldots,m+1} ,
    \label{eq:structure-Iin} \\
    & \bigcap_{j=1}^m I_{j,n} \eid \left(\min \bigcap_{j=1}^m I_{j,n} + \mathsf{Range}_\betast \right) \cap \{0,\ldots, n\}.
    \label{eq:structure-Iin-m-intersection}
\end{align}
The above two observations are needed  in the following proofs.

\vspace{1em}

\begin{proof}[Proof of Proposition \ref{prop:more-intersection-properties}]

\phantom{blank}

    \textbf{Part (i).}  
    Fix any $N \in \bbn$. 
    Let us consider indices in groups $\{km+1,\ldots,(k+1)m\} $ for $k=0,\ldots,K$.
    It follows that
    \begin{align*}
        & \PP \left(  \exists \, 1\leq j_{1} < \cdots <  j_m \leq Km  \text{ such that } nB \cap  \bigcap_{i=1}^m I_{j_i,n}  \neq \emptyset   \right) \\
        \geq & \PP \left(  \exists \, k=0,\ldots,  N - 1   \text{ such that } nB \cap  \bigcap_{i=1}^m I_{km+i,n}  \neq \emptyset   \right) = 1 - q_n ^K,
          \end{align*}
        where $  q_n \triangleq \PP( nB \cap I_{1,n} \cap \cdots \cap  I_{m,n} = \emptyset ) $.
    By  \eqref{eq:probability-non-empty-m-intersections} in Lemma  \ref{lem:intersections-from-samorodnitsky-wang-2019AoP} $(\rom2)$, $\lim_{K\to \infty} \lim_{n\to \infty} q_n^K = 0$, which finishes the proof.

\vspace{1em}

    \textbf{Part (ii).}  We apply the union bound twice to get that  
    \begin{align*}
        & \PP\left( \exists \, 1\leq 
       j_{1,n} < \cdots < j_{m+1,n} < n^\delta \text{ such that }\bigcap_{i=1}^{m+1} I_{j_i,n}  \neq \emptyset  \right) \\
       \leq & { \lfloor n^\delta \rfloor \choose m+1}  \cdot \PP\left( I_{1,n} \cap \cdots \cap I_{m+1,n} \neq \emptyset \right) \\
       \lesssim &  n^{(m+1)\delta}   \cdot (n+1) \cdot \PP\left( 0 \in I_{1,n} \cap \cdots \cap I_{m+1,n} \right) 
       \lesssim    
       \frac{ n^{(m+1)\delta + 1}}{ w_n^{m+1} } \to 0.
    \end{align*}
    The last limit is zero because $w_n \in \mathsf{RV}_{1-\beta}$ and   $(m+1) \delta + 1 - (m+1)(1-\beta) <0$.

\vspace{1em}

\textbf{Part (iii).} 
From  \eqref{eq:structure-Iin-m-intersection}, 
 we obtain the stochastic dominance 
\begin{equation}
    \label{eq:stochastic-dominance-cardinatlity-intersection}
    \left|  \bigcap_{j=1}^m I_{j,n} \right| \leq_P 
    \left|
     \mathsf{Range}_\betast \cap \{0,\ldots,n \}
    \right|.
\end{equation}
Then  \eqref{eq;m-intersections-cardinality-control;log-power-law-tail} follows from Lemma \ref{lem:cardinality-estimate-range-of-S-beta} $(\rom2)$ since 
$\vartheta_n \asymp 1 \big/ \overline{F_\betast} (n) \asymp U_\betast (n)$.

\vspace{1em}

\textbf{Part (iv).} We proceed  \eqref{eq:structure-Iin-m-intersection} to note that 
\begin{align*}
     \left| (0,nb) \cap \bigcap_{i=1}^m  I_{i,n} \right| \eid &  \left| \mathsf{Range}_\betast \cap  \left\{  0,\ldots, \lfloor nb \rfloor - \min \bigcap_{j=1}^m I_{j,n}  \right\} \right|  \\
     = &  (S_\betast )^\leftarrow (n \Theta_n),
\end{align*}
where 
$$
 \Theta_n \triangleq  \max \left \{ \frac{\lfloor nb \rfloor}{n} - \frac{ \min \bigcap_{j=1}^m I_{j,n}}{n}, 0 \right\}.
$$
By     \eqref{eq:convergence-of-first-simultaneous-intersection-time}, it direct to show that 
$$
\Theta_n
\xLongrightarrow{ \PP \left( \cdot \, \big| \, (0,nb) \cap \bigcap_{i=1}^m  I_{i,n} \neq \emptyset \right)}
b( 1- Q_\betast).
$$
Notice that 
 $\{ \mathscr{Z}_\betast (t):t \in \bbr_+  \}$  is a.s.\ continuous,  therefore the definition of $J_1$-topology implies 
 $$
 \frac{ (S_\betast )^\leftarrow (n \Theta_n)}{\vartheta_n} 
 \Rightarrow 
 \mathscr{Z} \big( b( 1- Q_\betast) \big) \eid b^\betast \mathscr{Z} ( 1- Q_\betast).
 $$
\end{proof}

\vspace{1em}

\begin{proof}[Proof of Theorem \ref{thm:key-preliminary-theorem-far-away-intersection}]
    \phantom{blank}

 Recall that the dynamical system $(\mathbb{Y},\mathcal{Y},\mu,\tau)$ is measure-preserving. Therefore, for any $\delta\in \bbr$ such that $\delta+B\subset [0,1]$, we shall have $p_n(B)=p_n(B+\delta)$. 
 It suffices to take $B=(0,b)$ or $[0,b)$. We take $B=[0,b)$ because two proofs are almost identical.  

 \vspace{1em}

 To establish the weak convergence in the first entry, we first 
 apply the last intersection decomposition. Namely, 
  \begin{align*}
     p_n (B) 
     = & \sum_{ a\in nB \cap \bigcap_{j=1}^m I_{j,n} } \PP\left( a \in I_{m+1,n} \, | \, \mathscr{F}_{n,B} \right) \\
     & \phantom{  \sum_{ a\in nB \cap \bigcap_{j=1}^m I_{j,n} } } \cdot \PP\left(  a = \max I_{m+1,n} \cap  nB \cap \bigcap_{j=1}^m I_{j,n} \, \bigg| \,  a \in I_{m+1,n}, \mathscr{F}_{n,B} \right) \\
     = & \frac{1}{w_n} \sum_{ a\in nB \cap \bigcap_{j=1}^m I_{j,n} }  \PP\left(  a = \max I_{m+1,n} \cap  nB \cap \bigcap_{j=1}^m I_{j,n} \, \bigg| \,  a \in I_{m+1,n}, \mathscr{F}_{n,B} \right).
 \end{align*}
 With respect to \eqref{eq:structure-Iin}, \eqref{eq:structure-Iin-m-intersection} and the strong Markov property of $\boldsymbol{S}_{\beta,m+1}$, we can rewrite the last identity as
 \begin{align*}
     p_n (B) = & \frac{1}{w_n}\sum_{ a\in nB \cap \bigcap_{j=1}^m I_{j,n} } \PP\left( \boldsymbol{S}_{\beta,m+1} \text{ escapes }  nB \cap \bigcap_{j=1}^m I_{j,n} \, \bigg| \, S_{\beta,m+1}(0)=a, \mathscr{F}_{n,B} \right).
 \end{align*}
Therefore, on $(  \Omega_{n,B}, \mathscr{F}_{n,B}, \boldsymbol{P}_{n,B})$, we have 
\begin{equation}
    \label{eq:pnB-as-random-capacity}
     p_n (B) \eid 
      \frac{1}{w_n} \mathsf{Capacity} \left( \left \{0,\ldots,\lfloor nb \rfloor - \min \bigcap_{j=1}^m I_{j,n} \right \} \cap \mathsf{Range}_\betast    \right) 
\end{equation}
Combine   \eqref{eq:pnB-as-random-capacity}, \eqref{eq;m-intersections-cardinality-weak-convergence} and 
\eqref{eq:convergence-of-range-capacities-subadditive-argument}, 
we obtain
\begin{equation}
     \label{eq:key-preliminary-joint-far-away-intersection-first-coordinate}
     \frac{ w_n p_n (B) }{ \vartheta_n}
   \Rightarrow
 p_\mathsf{escape} b^\betast \mathscr{Z}_\betast (1-Q_\betast).
\end{equation}

\vspace{1em}

 To prove the jointly weak convergence,  we choose product interval
 $(c_1,c_2]\times (c_3,c_4]\in (0,\infty)^2$. It the follows that 
    \begin{align*}
        & \PP\left(  \frac{p_n(B) w_n }{\vartheta_n} \in (c_1,c_2] , \ell_{1,n} p_n \in (c_3,c_4]  \right) \\
        = & \PP\left(  \frac{p_n(B) w_n }{\vartheta_n} \in (c_1,c_2]  \right)
        \left(   
        (1-p_n (B))^{c_3/p_n} -  (1-p_n(B))^{c_4/p_n}
        \right) \\
        \to & \PP \big( c_1<   p_\mathsf{escape} b^\betast \mathscr{Z}_\betast (1-Q_\betast) \leq c_2 \big) \PP(c_4 < \Gamma \leq c_4 ).
    \end{align*}
The limit holds because of \eqref{eq:key-preliminary-joint-far-away-intersection-first-coordinate}
and  that $p_n(B)$ a.s.\ uniformly goes to zero over  events $\left\{  \frac{p_n(B) w_n }{\vartheta_n} \in (c_1,c_2] \right \}$. We thus establish the joint convergence due to the  $\pi$-$\lambda$ monotone class theorem.
\end{proof}

\section{Extremal Limit Theorems}
\label{sec:extremal-limit-theorems}

\subsection{Formulation of Main Results}
\label{subsec:formulation-main-results}

We first formulate the assumption for the main results.

\vspace{1em}

\begin{assumption}
\label{assumption:main-assumptions-for-the-main-theorem}
\phantom{blank}
\begin{itemize}
    \item[$(\rom1)$]  Let $(\mathbb{Y},\mathcal{Y},\mu,\tau)$ be the dynamical system in \S \ref{subsec:DNRMC-construction}. 
    Assume that the memory parameter
    $$    \beta \in \left( \frac{m-1}{m}, \frac{m}{m+1} \right) , \quad \exists \, m \in \bbn, \quad m \geq 2   $$
    and the renewal epoch distribution $F_\beta$ satisfies Assumption \ref{assumption:memory-parameter-and-return-probabilities}.
    \item[$(\rom2)$]  Let $\mathcal{N}$ be an (independently scattered) infinitely divisible random measure on $(\mathbb{Y}, \mathcal{Y})$ with a constant local characteristic triple $(\sigma^2,\nu,b)$ controlled by $\mu$, see \cite{samorodnitsky:2016:book} \S 3.2. Assume that the tail L\'evy measure $\nu(x,\infty)$ is moderately heavy. In the context of Definition \ref{def:exp-slow-vary-distr}, we write
    \begin{equation}
      \label{eq:assumption-on-levy-tail-nu}
       \nu(x,\infty) = c_\# \overline{H_\#} (x) = c_\# \exp \left( -\int_{1}^x \frac{du}{h(u)}  \right), \quad x_0=1.
    \end{equation}
\end{itemize}
\end{assumption}
Under the Assumption \ref{assumption:main-assumptions-for-the-main-theorem}, the stationary sequences $(X_t: t\in \bbz)$ of the form 
\begin{equation}
\label{eq:sid-sequence}
    X_t = \int_{\mathbb{Y}} \mathbf{1}_A \circ \tau ^t (y )  \mathcal{N} (dy), \quad A = \left\{ y = (y_i)_{i\in \bbz} \in \mathbb{Y}: y_0 = 0 \right \} .
\end{equation}

\vspace{1em}

\begin{remark}
    \label{remark:sid-process}
    \phantom{blank}
    {\rm
    \begin{itemize}
        \item  The process $(X_t: t\in \bbz)$ is indeed stationary and infinitely divisible, see \cite{samorodnitsky:2016:book}. 
        \item  We choose an indicator function $\one_A$ in the stochastic integration, which helps to highlight the main spirit of all arguments. The replacement \eqref{eq:long-range-dependence-and-zero-extremal-index} by a general  function $f$ will not change the proof essentially. In the same logic, we take $x_0=1$ and $\overline{\nu} = c_\# \overline{H_\#}$ instead of $x_0>0$ and $\overline{\nu} \sim c_\# \overline{H_\#}$.
        \item The assumption that  $\beta \in \big( \frac{m-1}{m}, \frac{m}{m+1} \big)$ is important. Our proofs break down on boundary points $\big\{  \frac{m}{m+1}: m\in \bbn \big \}$. We conjecture that a different mode of extremal clustering will appear. 
    \end{itemize}
    }
\end{remark}

\vspace{1em}

 Let $V$ be the function 
 \begin{equation}
     \label{eq:function-V}
      V(z) = 
    \begin{cases}
    \left( 1 / \overline{\nu} \right)^\leftarrow (z),  \quad & z > z_0, \; z_0 \triangleq 1/ \nu(1,\infty) = 1/c_\#, \\
    0, & \text{otherwise}
    \end{cases}.
 \end{equation}
 Set $V_\# = (1 / \overline{H_\#})^\leftarrow$. Due to Assumption \ref{assumption:main-assumptions-for-the-main-theorem}, it follows that 
\begin{equation}
\label{eq:main-assumption-V-sharp-and-V}
    V (z) \sim V_\# (z), \quad V_\# ^\prime (z) \asymp \left[ \exp \left(  \int_{z_0}^z \frac{\zeta(\log u)}{u \log u} du   \right) \right]^\prime
\end{equation}
where notations have the same meanings as used in Definition \ref{def:exp-slow-vary-distr}.
With respect to the above functions $V$  \eqref{eq:function-V} and $h$   \eqref{eq:assumption-on-levy-tail-nu}, and the wandering rate sequence $(w_n: n \in \bbn)$  \eqref{eq:definition-wandering-rate}, and the sequence $(\vartheta_n: n\in \bbn)$ \eqref{eq:sequence-vartheta-n} , we  define normalizing constants 
    \begin{equation}
        \label{eq:extremal-limit-theorems-normalizing-constants}
        b_n = m V(w_n) + V(\vartheta_n), \quad a_n = h \circ V(w_n).
    \end{equation}

\vspace{1em}

\begin{theorem}[Convergence of Random Sup-Measures]
    \label{main-theorem:random-sup-measure-convergence}
    \phantom{blank}
    
    Let $(X(t): t\in \bbz)$ be the sequence  \eqref{eq:sid-sequence} under Assumption \ref{assumption:main-assumptions-for-the-main-theorem} , $(\mathcal{M}_n: n \in \bbn)$ be the empirical random sup-measures  \eqref{eq:stationary-sequence-empirical-rsm} and $\mathcal{M}_\beta$ be the limit random sup-measure  \eqref{eq:limit-rsm-restriction-unit-interval}. If $(a_n, b_n)$ are given by \eqref{eq:extremal-limit-theorems-normalizing-constants}, then, in the space $ \mathsf{SM}([0,1])$,
    \begin{equation}
    \label{eq:main-theorem-convergence-rsm}
        \frac{\mathcal{M}_n(\cdot) - b_n}{a_n} \Rightarrow \mathcal{M}_\beta(\cdot).
    \end{equation}
\end{theorem}

\vspace{1em}

\begin{theorem}[Convergence of Extremal Processes]
    \label{main-theorem:extremal-processes-convergence}
    \phantom{blank}
    
    Let the setup be the same as Theorem \ref{main-theorem:random-sup-measure-convergence}. For each $n\in \bbn$, let $\mathcal{E}_n$ be the empirical extremal process
    $$
    \mathcal{E}_n (t) = \max\{ X_i: 0\leq i \leq \lfloor nt \rfloor\}
    $$
    and let $\mathcal{E}_\beta$ be the limit extremal process \eqref{eq:limit-extremal-process}. Then,  in the space $D(0,1)$ equipped with the $J_1$-topology, 
    \begin{equation}
    \label{eq:main-theorem-convergence-ep}
        \frac{\mathcal{E}_n(\cdot) - b_n}{a_n} \Rightarrow \mathcal{E}_\beta(\cdot).
    \end{equation}
\end{theorem}

\vspace{1em}

\subsection{Series Representation}
\label{subsec:series-representation}

We shall replace the original sequence $(X_t:t\in \bbz)$ by an compound Poisson sequence without changing the extremal behavior. This is the idea of the  series representation.

\vspace{1em}

Let $\mathcal{N}^{(1)}$ and $\mathcal{N}^{(2)}$ be two independent infinitely divisible random measures such that 
 $\mathcal{N}^{(1)}$ has the constant local characteristic triple $(\sigma^2,   [\nu]_{(-\infty,1]}, b-\nu(1,\infty))$ and
    $\mathcal{N}^{(2)}$ has the constant local characteristic triple $(0,[\nu]_{(1,\infty)},\nu(1,\infty) )$.
Clearly, $\mathcal{N}\eid \mathcal{N}^{(1)} + \mathcal{N}^{(2)}$ and $\big( X_t : t\in \bbz \big) \eid   \big(X^{(1)}_t : t\in \bbz \big) +   \big( X^{(2)}_t : t\in \bbz \big)$ where
\begin{equation}
    X^{(i)}_t \triangleq \int_{\mathbb{Y}} \mathbf{1}_A \circ \tau ^t (y ) \, \mathcal{N}^{(i)} (dy), 
   \quad t \in \bbz, \quad i=1,2. \label{eq:light-tailed-heavy-tailed-component} 
\end{equation}

\vspace{1em}

\begin{proposition}[Super-Exponential Component]
\label{prop:light-components-negaligible}
\phantom{blank} 

Let $ \big(X^{(1)}_t : t\in \bbz \big) $ be in \eqref{eq:light-tailed-heavy-tailed-component}, then
\begin{equation}
    \label{eq:light-tailed-component-is-negligible}
     \frac{\max_{0\leq t \leq n } X^{(1)} _t }{a_n} \xlongrightarrow{P} 0 .
\end{equation}
\end{proposition}

\begin{proof}[Proof of Proposition \ref{prop:light-components-negaligible}]
    \phantom{blank}
    
    Due to the union bound, it suffices to show that, for any $\epsilon>0$, 
    $$
    n \PP\left(  X^{(1)}_0> \epsilon a_n \right)  \to 0.
    $$
    Because the L\'evy measure of $ X^{(1)}_0$ is bounded on the right, we have
    $\PP (  X^{(1)} (0)  > x ) = o(e^{-cr}) $ for any $c>0$; see Theorem 26.1 in \cite{sato:2013:book}. We also note that $a_n \geq (\log n)^2, n\gg 0$ due to Proposition \ref{prop:properties-exp-slowly-varying-distr}  $(\rom1)$. By taking $c=1$, we get 
    $$
     n\PP\left(   X^{(1)}  (0) > \epsilon a_n \right) \leq n \exp \left(  - (\log n)^2 \right) \to 0.
    $$
\end{proof}

\vspace{1em}

     Proposition \ref{prop:light-components-negaligible} implies that we should analyze the extremal behavior of  $ (X^{(2)}_t : t\in \bbz ) $. To prove Theorems \ref{main-theorem:random-sup-measure-convergence} and \ref{main-theorem:extremal-processes-convergence}, it suffices to replace $(X_t: t\in \bbz)$ by $(X_t^{(2)}: t\in \bbz)$.
     For each $n\in \bbn$, the random vector $\big( X^{(2)}_0,\ldots,  X^{(2)}_n \big)$ admits a \textit{series representation}. Namely, 
\begin{equation}
        \left( X^{(2)}_t \right)_{0 \leq t \leq n} \eid \left( \sum_{j=1}^\infty V \left( \frac{w_n}{\Gamma_j} \right) \mathbf{1} (t \in I_{j,n})  \right)_{0 \leq t \leq n}, 
    \label{eq:series-representation} 
\end{equation}
where the function $V$ is given in \eqref{eq:function-V},  the sequence $(\Gamma_j: j\in \bbn)$ consists  of arrival times of a unit rate Poisson process,  $(I_{j,n}: j\in \bbn)$ is defined in  \eqref{eq:return-times-random-zeros-Ikn}. In addition, families  $(\Gamma_j: j\in \bbn)$  and $(I_{j,n}: j\in \bbn)$ are independent. To prove the series representation, the key is show that the two sides of   \eqref{eq:series-representation} are both compound Poisson and share the same characteristic function, see Theorem 3.3.10 and Theorem 3.4.1 in \cite{samorodnitsky:2016:book}.
Note that all subsequent proofs will only involves partial sequence  $(X_t: 0 \leq t \leq n)$ instead of the whole. Thus, for our convenience, \textbf{we shall hereafter identify}
\begin{equation}
    \label{eq:series-representation-abuse}
    \big( X_t \big)_{0 \leq t \leq n} = \left( \sum_{j=1}^\infty V \left( \frac{w_n}{\Gamma_j} \right) \mathbf{1} (t \in I_{j,n})  \right)_{0 \leq t \leq n}.
\end{equation}

\vspace{1em}

 Under this convention, we enumerate the non-zero summands in $X_t$ by indices $j_1(t)< j_2(t)<\ldots$, i.e.,
\begin{equation}
    \label{eq:enumeration-non-vanishing-summands-in-Xt}
     X(t) = \sum_{i\geq 1}  V \left( \frac{w_n}{\Gamma_{j_i(t)} } \right)
     \quad 
     \text{and}
     \quad
     t\in \bigcap_{i\geq 1}I_{j_i(t),n}.
\end{equation}
In particular, we shall use the following equivalent formulation
\begin{equation}
    \label{eq:series-representation-X_t-equivalent-formulation}
    X_t \text{ contains }  V \left( \frac{w_n}{\Gamma_{j_i} } \right) \text{ for some } i=1,\ldots,k \quad 
\Longleftrightarrow \quad 
t \in \bigcap_{i=1}^k I_{j_i, n}.
\end{equation}
We shall also need to obvious results. For any $j \in \bbn$, we combine 
the $\Pi$-variation propery \eqref{eq:Pi-varying-definition} and the fact $h(x)=o(x)$ to get
\begin{equation}
    \label{eq:Pi-variation-and-weak-convergence}
    \frac{V \left(  w_n  / \Gamma_j  \right) - V(w_n) }{  h\circ V(w_n)  } 
    \Rightarrow - \log \Gamma_j \quad 
    \text{and} \quad 
     \frac{V \left(  w_n / \Gamma_j  \right)  }{  V(w_n)  } \Rightarrow 1.
\end{equation}

\subsection{Lower Bounds}
\label{subsec:lower-bound}

For each open interval $B \subset [0,1]$, we shall construct explicitly a lower bound $M_n^\mathsf{lower}(B) \leq \mathcal{M}_n(B)$, which consists of
 $m$ dominating summands of magnitude $O_P(V(w_n))$ and an extra term of magnitude  $O_P(V(\vartheta_n))$.

\vspace{1em}

Let us fix a  $\delta_0 \in (0,1-\beta - 1/(m+1))$ throughout in \S \ref{subsec:lower-bound}. We introduce
\begin{equation}
     \mathcal{J}_n( B) = \left\{
     (j_1,...,j_m)_<  \in \bbn^m:  j_m \leq \lfloor n^{\delta_0} \rfloor , nB \cap \bigcap_{i=1}^m I_{j_i, n}  \neq \emptyset
    \right\}
    \label{eq:m-intersected-sets-collection-prelimit}
\end{equation}
and
\begin{align}
     J^\star_n(B) = & (j_1^\star(B), \ldots, j_m^\star (B) )_<   \notag \\
     = & 
     \begin{cases}
         \argmax_{(j_i)_< \in  \mathcal{J}_n( B)} \sum_{i=1}^m
     V \left( \frac{w_n}{\Gamma_{j_i} } \right)  \quad & \mathcal{J}_n(B) \text{ is non-empty} \\
     (\infty,\ldots,\infty) & \text{otherwise}
     \end{cases}  .
    \label{eq:optimal-m-intersected-sets-prelimit} 
  \end{align}
With respect to $ J^\star_n(B)$, we further introduce 
\begin{equation}
    \label{eq:the-extra-term-index}
     j^{\doublestar}_{m+1}(B)= 
     \begin{cases}
         \inf \left\{  j > j^\star_m(B) :  I_{j,n} \cap nB \cap \bigcap_{i=1}^m I_{j_i^\star(B), n} \neq \emptyset   \right\}  \quad & j^\star_m(B) < \infty \\
     \infty  & j^\star_m(B) = \infty
     \end{cases}.
\end{equation}
We define
\begin{equation}
    \label{eq:the-lower-bound-rsm}
      M_n^\mathsf{lower}(B) =
      \sum_{i=1}^m V \left( \frac{w_n}{\Gamma_{j^\star_i (B) }} \right)       
      + V \left(  \frac{w_n}{ \Gamma_{ j^\doublestar_{m+1}(B) } } \right) 
\end{equation}
and it is obvious that 
\begin{equation}
     \label{eq:the-lower-bound-rsm-is-trivial}
     \lim_{n\to \infty}
     \PP \left(  \mathcal{M}_n (B) \geq M_n^\mathsf{lower}(B)  \right) = 1.
\end{equation}
because of the construction in \eqref{eq:optimal-m-intersected-sets-prelimit} and \eqref{eq:the-extra-term-index}.
\vspace{1em}

We recall the random closed sets 
$R_j, j\in \bbn$ in \eqref{eq:stationary-srs-unit-interval}. Based on the representation \eqref{eq:limit-rsm-restriction-unit-interval} of the random sup-measure $\mathcal{M}_\beta$, we define
\begin{align}
    & \mathcal{L}(B) = \left\{
      (\ell_1,...,\ell_m)_<  \in \bbn^m:  B \cap \bigcap_{i=1}^m R_{\ell_i}  \neq \emptyset
    \right\}, 
    \label{eq:m-intersected-sets-collection-prolimit}\\
     & L^\star(B) = (\ell_1^\star(B), \ldots, \ell_m ^\star (B) )_< =\argmax_{(\ell_i)_< \in  \mathcal{L}(B)} \sum_{i=1}^m
     -\log \Gamma_{\ell_i}.
    \label{eq:optimal-m-intersected-sets-prolimit}
   \end{align}
 The vector $L^\star(B)$ is a.s.\ $\bbn^m$-valued  because $\mathcal{L}(B)$ is a.s.\ non-empty. This follows from Lemma \ref{lem:zero-positive-dicotomy} and that  $-\log \Gamma_\ell \xlongrightarrow{a.s.} -\infty$ as $\ell \to \infty$.

\vspace{1em}

\vspace{1em}

\begin{proposition}
    \label{prop:lower-bound}
    \phantom{blank}

    \begin{itemize}
        \item[$(\rom1)$] 
      Let $s\in \bbn$  and  $B_1,\ldots, B_s \subseteq [0,1]$ be any disjoint open intervals, then  
    \begin{equation}
    \label{eq:key-preliminary-weak-convergence-optimal-to-optimal}
    \big(  J^{\star}_{n} ( B_1), \ldots, J^{\star}_{n} ( B_s)  \big)
    \Rightarrow
    \big(  L^{\star} ( B_1), \ldots, L^{\star} (B_s)  \big).
     \end{equation}
     Thus, $V (  w_n/ \Gamma_{j^\star_i (B)} )=O_P(w_n)$ for all $i=1,\ldots,m$.
    \item[$(\rom2)$] Let $(w_n:n\in \bbn)$ and $(\vartheta_n:n\in \bbn)$ be sequences \eqref{eq:definition-wandering-rate} and \eqref{eq:sequence-vartheta-n} respectively. Then, for any open interval $B$, 
    \begin{equation}
        \label{eq:weak-convergence-of-j-star-star-m-plus-one}
        j^\doublestar_{m+1}(B)  \frac{\vartheta_n}{ w_n} \Rightarrow U_B,
    \end{equation}
    where $U_B$ is some random variable distributed on $(0,\infty)$.
     \item[$(\rom3)$] Let $\mathcal{M}_\beta$ be the random sup-measure \eqref{eq:limit-rsm-restriction-unit-interval}. Then, for the $B_1,\ldots, B_s$ in part $(\rom1)$, 
      \begin{equation}
        \label{eq:weak-convergence-of-lower-bound}
         \left( \frac{M_n^\mathsf{lower} (B_r) - b_n}{a_n}
       \right)_{r=1,\ldots,s} \Rightarrow \big( \mathcal{M}_\beta(B_r) \big)_{r=1,\ldots,s}.
    \end{equation}
    \end{itemize}
\end{proposition}

\begin{proof}[Proof of Proposition \ref{prop:lower-bound}]
\phantom{blank}

\textbf{Part (i).}
The goal is to apply the ``Billingsley's convergence together lemma", see  Theorem 3.2 in \cite{billingsley:1999:book}, to prove \eqref{eq:key-preliminary-weak-convergence-optimal-to-optimal}. We define for each $k\in \bbn$
\begin{align*}
    & J^{\star}_{n} (k, B_r)  = \big( j^\star_{1}(B_r) \wedge k, \ldots,  j^\star_{m}(B_r) \wedge k \big),  \quad r=1,\ldots,s \\
   &  L^{\star}( k, B_r) =  \big( \ell ^\star_{1}(B_r) \wedge k, \ldots,  \ell ^\star_{m}(B_r) \wedge k \big), \quad r=1,\ldots,s.
\end{align*}
We apply \eqref{eq:Pi-variation-and-weak-convergence} and  \eqref{eq:convergence-of-intersections-samorodnitsky&wang;1}  to get 
\begin{align*}
    & \bigg(   \left(  \frac{V(w_n/\Gamma_j) - V(w_n)}{h\circ V(w_n)}   \right)_{j=1,\ldots,k},  \Big( \bigcap_{j\in \mathscr{J}} I_{j,n} \Big)_{\mathscr{J} \subset \{1,\ldots, k\} }  \bigg) \\
    \Rightarrow & \bigg(  \left( - \log \Gamma_j   \right)_{j=1,\ldots,k}, \Big( \bigcap_{j\in \mathscr{J}} R_j \Big)_{\mathscr{J} \subset \{1,\ldots, k\} }   \bigg).
\end{align*}
It thus follows that 
\begin{equation}
    \label{eq:lower-bound-billingsley-convergence-together-condition-1}
    \left(  J^{\star}_{n} (k, B_1), \ldots, J^{\star}_{n} ( k,B_s)  \right)
    \xLongrightarrow{n\to \infty}
    \left(  L^{\star} (k, B_1), \ldots, L^{\star} (k,B_s)  \right).
\end{equation}
Because each random vector $ L^{\star} (B_r)$ is a.s.\ $\bbn^m$-valued,  we immediately obtain 
\begin{equation}
    \label{eq:lower-bound-billingsley-convergence-together-condition-2}
      \left(  L^{\star} (k, B_1), \ldots, L^{\star} (k,B_s)  \right)
        \xlongrightarrow[a.s.]{k\to \infty}
         \left(  L^{\star} ( B_1), \ldots, L^{\star} (B_s)  \right).
\end{equation}
Besides \eqref{eq:lower-bound-billingsley-convergence-together-condition-1} and \eqref{eq:lower-bound-billingsley-convergence-together-condition-1}, it remains to show
\begin{equation}
 \label{eq:lower-billingsley-convergence-together-condition-3}
     \lim_{k \to \infty}
    \limsup_{n \to \infty}
    \PP
    \Big( 
     J^\star_n (k, B) = J^\star_n (B)  
    \Big) = 0.
\end{equation}
to trigger the ``convergence together lemma".

\vspace{1em}
Consider the event 
$$
     E_{K,n} = \left\{ \exists \,  1\leq j_1 < \cdots < j_m \leq K  \text{ such that }  nB \cap \bigcap_{i=1}^m I_{j_i,n}  \neq \emptyset \right \} .
$$
On the event $ E^{(1)}_{K,n}$, $\mathcal{J}_n(B)$ is clearly non-empty and hence, 
\begin{equation}
  \label{eq:tightness-of-random-sum;lower-bound;1}
   \sum_{i=1}^m V \left(  \frac{w_n}{\Gamma_{j^\star_i (B)}} \right) > m V\left(  \frac{w_n}{\Gamma_K} \right).
\end{equation}
From Proposition \ref{prop:more-intersection-properties}, we obtain that  
$$
\lim_{K\to \infty} \limsup_{n\to \infty} \PP\left(  E_{K,n}   \right)=1,
$$
therefore,
\begin{equation*}
    \lim_{K\to \infty} \limsup_{n\to \infty} 
\PP 
\left(
 \sum_{i=1}^m V \left(  \frac{w_n}{\Gamma_{j^\star_i (B)}} \right) > m V\left(  \frac{w_n}{\Gamma_K} \right)
\right) = 1.
\end{equation*}
It is direct to notice that 
$$
\sum_{i=1}^m V \left(  \frac{w_n}{\Gamma_{j^\star_i (B)}} \right) < 
 (m-1) V\left(  \frac{w_n}{\Gamma_1} \right) + V \left(  \frac{w_n}{\Gamma_{j^\star_m (B)}} \right).
$$
With respect to any fixed $K$, it follows that
\begin{align*}
   & \frac{ (m-1) V\left(  \frac{w_n}{\Gamma_1} \right) + V \left(  \frac{w_n}{\Gamma_N} \right) - m V\left(  \frac{w_n}{\Gamma_K} \right)  }{ h\circ V(w_n) } \\
& \xLongrightarrow{n\to\infty}  m \log \Gamma_K    - (m-1) \log \Gamma_1 - \log \Gamma_N   \xLongrightarrow{N\to\infty}  -\infty
\end{align*}
To avoid contradiction, we claim that
\begin{equation*}
     \lim_{N \to \infty}
    \limsup_{n \to \infty}
    \PP
    \Big( 
    j^\star_m(B) < N 
    \Big) = 1,
\end{equation*}
which immediately implies \eqref{eq:lower-billingsley-convergence-together-condition-3}.

   \vspace{1em}

   \textbf{Part (ii).} 
   We define a new random object $I_{0,n}$ such that $I_{j,n}, j=0,1,2,\ldots$ are i.i.d. We note that, conditioned on $J^\star_n(B)$ and $I_{j_i^\star (B), n}, i=1,\ldots,m$, 
 $j^\doublestar_{m+1} (B) - j^\star_m(B)$ is geometrically distributed with (random) probability $p^\star_n(B)$,
 \begin{align*}
     p^\star_n(B) & =  \PP \left( I_{0,n}  \cap nB \cap \bigcap_{i=1}^m I_{j_i^\star(B),n}  \neq \emptyset \, \bigg| \, J^\star_n(B), I_{j_i^\star (B), n}, i=1,\ldots,m \right) \\
     & \eid  p_n(B),
 \end{align*}
where $p_n(B)$ is the (random) probability \eqref{eq:random-intersection-probability-conditional-probability}. 
Since $j^\star_m(B)$ is tight by part $(\rom1)$, we thus see that 
$ j^\doublestar_{m+1} (B) p_n^\star(B) $ has the same distributional limit as the $\ell_n(B) p_n(B)$ in 
\eqref{eq:key-preliminary-joint-far-away-intersection}. We thus derive  \eqref{eq:weak-convergence-of-j-star-star-m-plus-one} from Theorem \ref{thm:key-preliminary-theorem-far-away-intersection}.

   \vspace{1em}

   \textbf{Part (iii).}  
  From parts $(\rom1)$ of this proposition, we note that  
   $$
   \Bigg(
   \frac{\sum_{i=1}^m V \left(  \frac{w_n}{\Gamma_{j^\star_i (B)}} \right) - b_n}{a_n}
   \Bigg)_{r=1,\ldots,s} \Rightarrow
   \big(\mathcal{M}_\beta(B_r) \big)_{r=1,\ldots,s}
   $$
   To prove\eqref{eq:weak-convergence-of-lower-bound}, it remains to show that, for any open interval $B \subset [0,1]$, 
   \begin{equation}
       \label{eq:lower-bound-drift-goes-to-zero}
         \frac{ 
  V \left(  \frac{w_n}{ \Gamma_{ j^\doublestar_{m+1}(B) } } \right) 
  - V(\vartheta_n)
 }
 {a_n} \Rightarrow 0.
   \end{equation}
 From the  Law of Large Numbers and part $(\rom2)$, we have
$$
 \frac{V \left(  \frac{w_n}{ \Gamma_{ j^\doublestar_{m+1}(B) } } \right) 
  - V(\vartheta_n)
 }
 {h\circ V(\vartheta_n)} \Rightarrow -\log U.
 $$
By Proposition \ref{prop:properties-exp-slowly-varying-distr} $(\rom2)$, we also have   
$$
\frac{h \circ V(\vartheta_n)}{a_n} = \frac{h \circ V(\vartheta_n)}{h \circ V(w_n)} \to 0,
$$
because  $\vartheta_n \in \mathsf{RV}_\betast$, $w_n \in \mathsf{RV}_{1-\beta}$ and $\betast < 1-\beta$.
Hence we get  \eqref{eq:lower-bound-drift-goes-to-zero}.

\end{proof}

\subsection{Upper Bounds}
\label{subsubsec:upper-bound-rsm-proof}

For any open interval $B\subset [0,1]$, we define another bound as
\begin{equation}
    \label{eq:the-upper-bound-rsm}
     M_n^\mathsf{upper}(B, c_1, c_2) = 
      \sum_{i=1}^m V \left(  \frac{w_n}{\Gamma_{j^\star_i (B)}} \right)       +  d_n,
\end{equation}
where, for convenience, 
\begin{equation}
    \label{eq:extra-drift-term-upper-bound}
    d_n = d_n (c_1,c_2) = V \left(   c_1 \vartheta_n \right) + c_2 h \circ V \left(   c_1 \vartheta_n \right), \quad c_1 \geq 0, \quad c_2 \geq 0.
\end{equation}
Our goal is to show that
\begin{equation}
    \label{eq:upper-bound-with-high-P-no-exceed}
    \lim_{c_1, c_2 \to \infty} 
    \limsup_{n \to \infty}
    \PP 
    \left(
    \max_{t\in nB} X_t > M_n^\mathsf{upper}(B, c_1, c_2)
    \right) = 0.
\end{equation}
 According to the index $j_m(t)$ in \eqref{eq:enumeration-non-vanishing-summands-in-Xt}, we shall divide $X_t$'s into three classes: top, middle and bottom. The strategy is to properly use union bound methods in each class.

\subsubsection{Top Comparison}
\label{subsubsec:top-comparison}

For each $k_0 \in \bbn$, define
\begin{equation}
       M_n^\mathsf{top}(B, k_0) = \max \left\{ X_t : t\in nB , j_{m,n}(t) \leq 2^{k_0}  \right  \}.
     \label{eq:max-top-part;1} 
\end{equation}
According to the largest $m$ summands in the series representation,  we would like to rewrite $M_n^\mathsf{top}(B, k_0)$ in a different form. Hence we define  
\begin{equation}
  \label{eq:m-intersected-sets-collection-before2k0-prelimit}
      \mathcal{J}_n^\mathsf{top}( k_0, B) = \left\{
     (j_1,...,j_m)_<  \in \bbn^m:  j_m \leq 2^{k_0} \text{ and } nB \cap \bigcap_{i=1}^m I_{j_i, n}  \neq \emptyset
    \right\}
\end{equation}
and rewrite 
\begin{equation}
    \label{eq:max-top-part;2}
     M_n^\mathsf{top}(B, k_0)  =  \max_{  J  \in \mathcal{J}_n( k_0, B)   } \;
    \max_{ t \in nB \cap \bigcap_{i=1}^m I_{j_i, n} }   X_t, 
\end{equation}
 where $J \triangleq (j_1,\ldots,j_m)_<$ for convenience.
Due to \eqref{eq:finite-intersection-properties-empty}, with probability approaching one, any $X_t$ contains at most $m$ summands $V(w_n/\Gamma_j)$ for index $j\in \{1,\ldots, \lfloor n^{\delta_0} \rfloor\}$. Therefore, we define 
\begin{align*}
   M_n^\mathsf{top}(B, k_0, J) = & \sum_{i=1}^m V \left(  \frac{w_n}{\Gamma_{j_i}} \right) \\
     & + 
     \max_{ t \in nB \cap \bigcap_{i=1}^m I_{j_i, n} } \sum_{j > n^{\delta_0} } V \left(  \frac{w_n}{\Gamma_j} \right)  \one (t\in I_{j,n}).  
\end{align*}
And it easily follows that
\begin{equation}
    \label{eq:max-top-part;4}
   \lim_{n\to \infty} 
   \PP
   \left(
   M_n^\mathsf{top}(B, k_0) =  \max_{  J \in  \mathcal{J}_n^\mathsf{top}( k_0, B)   } 
   M_n^\mathsf{top}(B, k_0, J)
    \right) = 1
\end{equation}

\begin{proposition} 
\label{prop:top-part-upper-bound}
\phantom{blank}

\begin{equation}
\label{eq:top-part-not-exceed-upper-bound}
        \lim_{k_0 \to \infty } \,
         \lim_{ c_1, c_2\to \infty } \,
    \limsup_{n\to \infty}
    \PP
    \Big(
    M_n^\mathsf{top}(B,k_0)
      > M_n^\mathsf{upper}(B, c_1, c_2)
    \Big) = 0.
\end{equation}

\end{proposition}

\begin{proof}[Proof of Proposition \ref{prop:top-part-upper-bound}]
\phantom{blank}

Due to  \eqref{eq:max-top-part;4}, it suffices to show that 
\begin{equation}
\label{eq:top-part-not-exceed-upper-bound-equivalent}
        \lim_{k_0 \to \infty } \,
         \lim_{ c_1, c_2\to \infty } \,
    \limsup_{n\to \infty}
    \PP
    \Big(
    \max_{  J \in  \mathcal{J}_n^\mathsf{top}( k_0, B)   } 
   M_n^\mathsf{top}(B, k_0, J)
      > M_n^\mathsf{upper}(B, c_1, c_2)
    \Big) = 0.
\end{equation}
Fix any $k_0$. As long as $n^{\delta_0} > k_0$, one always have that 
$$
\max_{ J \in  \mathcal{J}_n^\mathsf{top}( k_0, B)  } \sum_{i=1}^m  V \left(  \frac{w_n}{\Gamma_{j_i}} \right) 
\leq  \sum_{i=1}^m  V \left(  \frac{w_n}{\Gamma_{j_i^\star(B)}} \right) 
$$
due to the construction of $J^\star_n(B)$.
Hence,  instead of  \eqref{eq:top-part-not-exceed-upper-bound-equivalent}, it suffices to show 
\begin{align}
        & \lim_{k_0 \to \infty } \,
         \lim_{ c_1, c_2\to \infty } \,
    \limsup_{n\to \infty}  \notag \\
    & \PP
    \bigg(  
    \max_{  J \in  \mathcal{J}_n^\mathsf{top}( k_0, B)   } \;
    \max_{ t \in nB \cap \bigcap_{i=1}^m I_{j_i, n} } 
    \sum_{j > n^{\delta_0} } V \left(  \frac{w_n}{\Gamma_j} \right)  \one (t\in I_{j,n}) 
   >  d_n
    \bigg) = 0.
     \label{eq:top-part-remainder-not-exceed-upper-bound-remainder}
\end{align}

\vspace{1em}

To this end, note first that $| \mathcal{J}_n( k_0, B) | \leq 2^{k_0m}$, hence the union bound method implies
\begin{align*}
      & \PP
    \Bigg(  
    \max_{  J \in  \mathcal{J}_n^\mathsf{top}( k_0, B)   } \;
    \max_{ t \in nB \cap \bigcap_{i=1}^m I_{j_i, n} } 
    \sum_{j > n^{\delta_0} } V \left(  \frac{w_n}{\Gamma_j} \right)  \one (t\in I_{j,n}) 
  > d_n 
    \Bigg) \\
    \leq & 2^{k_0m} 
    \cdot \PP
    \Bigg(  
    \max_{ t \in nB \bigcap_{i=1}^m I_{i, n} } 
    \sum_{j > n^{\delta_0} } V \left(  \frac{w_n}{\Gamma_j} \right)  \one (t\in I_{j,n}) 
  > d_n
  \, \bigg| \, nB \cap \bigcap_{i=1}^m I_{i, n} \neq \emptyset
    \Bigg) \\
      \leq & 2^{k_0m} 
    \cdot \PP
    \Bigg(  
    \max_{ t \in \bigcap_{i=1}^m I_{i, n} } 
    \sum_{j > n^{\delta_0} } V \left(  \frac{w_n}{\Gamma_j} \right)  \one (t\in I_{j,n}) 
  > d_n
  \, \bigg| \,  \bigcap_{i=1}^m I_{i, n} \neq \emptyset
    \Bigg) \\
    \leq &  2^{k_0m} 
    \cdot \left[  \PP
    \left(  
  \left| \bigcap_{i=1}^m I_{i, n} \right| > c_1 \vartheta_n
  \, \bigg| \,  \bigcap_{i=1}^m I_{i, n} \neq \emptyset
    \right) 
    +  c_1 \vartheta_n   
    \cdot \PP
    \left(  
    X_0 > d_n
    \right) \right].
\end{align*}
From  \eqref{eq;m-intersections-cardinality-weak-convergence}, we note that $\EE \left| \bigcap_{i=1}^m I_{i, n} \right| \lesssim \vartheta_n$. Therefore, the Markov inequality guarantees that 
\begin{equation}
    \label{eq:upper-bound-top-part-final-1}
   \lim_{c_1 \to \infty} \limsup_{n\to \infty} 2^{k_0m} \cdot \PP
    \left(  
  \left| \bigcap_{i=1}^m I_{i, n} \right| > c_1 \vartheta_n
  \, \bigg| \,  \bigcap_{i=1}^m I_{i, n} \neq \emptyset
    \right) =0.
\end{equation}
Next, due to $\PP(X_0 > d_n) \asymp \overline{nu}(d_n)$, it follows that as $n\to \infty$, 
\begin{align*}
    c_1 \vartheta_n   
    \cdot \PP
    \left(  
    X_0 > d_n
    \right)  \lesssim & \frac{ \overline{\nu} \big( V(c_1 \vartheta_n) + c_2 h \circ V(c_1 \vartheta_n)   \big)  }{  \overline{\nu} \big( V(c_1 \vartheta_n) \big)   } \lesssim e^{-c_2}.
\end{align*}
The last inequality is due to Assumption \ref{assumption:main-assumptions-for-the-main-theorem} and  \eqref{eq:Gamma-varying-definition}. Thus,
\begin{equation}
    \label{eq:upper-bound-top-part-final-2}
    \lim_{c_2 \to \infty} \limsup_{n\to \infty} 
    2^{k_0m} \cdot c_1 \vartheta_n   
    \cdot \PP
    \left(  
    X_0 > d_n
    \right)  =0
\end{equation}
Combining  \eqref{eq:upper-bound-top-part-final-1} and  \eqref{eq:upper-bound-top-part-final-2}  together gives
  \eqref{eq:top-part-remainder-not-exceed-upper-bound-remainder}.
\end{proof}

\subsubsection{Middle  Comparison}
\label{subsubsec:middle-comparison}

Let $\rho>0$  and set 
\begin{equation}
    \label{eq:sequence-k-n}
    k_n =  \left \lfloor \frac{\rho \log n}{\zeta(\log n)} \right \rfloor 
\end{equation}
Regardless of the open interval $B\subset [0,1]$, we define 
\begin{equation}
    \label{eq:upper-bound-mid-part-extreme}
    M_n^\mathsf{mid}=  M_n^\mathsf{mid}(k_0) =  \max \left\{ X_t : 0 \leq t  \leq n,  2^{k_0} < j_m(t) \leq  2^{k_n} \right\}.
\end{equation}
We shall again adopt the simplification $J= (j_1,...,j_m)_<$ and the $d_n$ in \eqref{eq:extra-drift-term-upper-bound}. Additionally, we shall set 
$$
\eta_n=\eta_n(c_3) \triangleq V(c_3 \vartheta_n \log n), \quad c_3 \geq 0.
$$
for further convenience. In the sequel, we shall replace the sum $\sum_{i=1}^m V \left( w_n / \Gamma_{j_i^\star(B)} \right) $ by a deterministic number. Namely,   from the proof of  \eqref{eq:tightness-of-random-sum;lower-bound;1} and the Law of Large Numbers, we obtain 
\begin{equation}
     \label{eq:a-deterministic-bound-m-largest-jumps}
      \lim_{K\to \infty} \lim_{n\to\infty} \PP \left(
    \sum_{i=1}^m V \left( \frac{w_n}{ \Gamma_{j_i^\star(B)} }\right)  
    > m V \left( \frac{w_n}{K} \right)
    \right) = 1.
\end{equation}

\vspace{1em}

\begin{proposition}
\label{prop:middle-part-upper-bound}

\begin{equation}
    \label{eq:middle-part-do-not-exceed-upper-bound}
\lim_{ \substack{\rho\to 0 \\  k_0 \to \infty } } 
    \limsup_{n \to \infty} 
     \PP 
    \bigg(  
   M_n^\mathsf{mid}
     > M^\mathsf{upper}_n (B,c_1=0,c_2=0)
    \bigg) = 0.
\end{equation}

\end{proposition}

\begin{proof}[Proof of Proposition \ref{prop:middle-part-upper-bound}]
    \phantom{blank}

Due to \eqref{eq:a-deterministic-bound-m-largest-jumps}, it suffices to show that, for any $K >0$, 
\begin{equation}
    \label{eq:middle-part-do-not-exceed-upper-bound;equivalent-1}
\lim_{ \substack{\rho\to 0 \\ k_0 \to \infty } } 
    \limsup_{n \to \infty} 
     \PP 
    \bigg(  
   M_n^\mathsf{mid}
     > m V\left( \frac{w_n}{K} \right) 
    \bigg) = 0.
\end{equation}
Let us decompose $M_n^\mathsf{mid} = \bigvee_{k_0 \leq k \leq k_n - 1} M_n^\mathsf{mid}(k)$, where
\begin{equation}
    M_n^\mathsf{mid}(k) \triangleq  \max \left\{ X_t : 0 \leq t  \leq n,  2^{k} < j_m(t) \leq  2^{k+1} \right\}.
     \notag 
\end{equation}
Thus, \eqref{eq:middle-part-do-not-exceed-upper-bound;equivalent-1} is equivalent to 
\begin{equation}
\lim_{ \substack{\rho\to 0 \\  k_0 \to \infty } } 
    \limsup_{n \to \infty} 
 \PP 
    \left(  \forall \, k=k_0,\ldots, k_n, \; 
   M_n^\mathsf{mid}(k)
     > m V\left( \frac{w_n}{K} \right) 
    \right) = 0.
     \label{eq:middle-part-do-not-exceed-upper-bound;equivalent-2}
\end{equation}

\vspace{1em}

Let us show \eqref{eq:middle-part-do-not-exceed-upper-bound;equivalent-2}. We define for each $k=k_0,\ldots,k_n-1$,
\begin{align}
    &  \mathcal{J}_n^\mathsf{mid}( k) = \left\{
     (j_1,...,j_m)_<  \in \bbn^m:  2^{k}+1 \leq j_m \leq 2^{k+1} \text{ and } \bigcap_{i=1}^m I_{j_i, n}  \neq \emptyset
    \right\}, \notag \\ 
    &   T_n(k) = \max_{ J \in \mathcal{J}_n^\mathsf{mid}( k) } \, \max_{t\in \bigcap_{i=1}^m I_{j_i,n}  }
    \sum_{j > n^{\delta_0} } V \left( \frac{w_n}{\Gamma_j}  \right) \one(t \in I_{j,n}). \notag
\end{align}
  By \eqref{eq:finite-intersection-properties-empty}, we note that, as $n\to \infty$, 
 \begin{align*}
    \PP 
     \bigg(& \forall \, k= k_0, \ldots, k_n-1, \;
     \\
    & M_n^\mathsf{mid}(k) <  (m-1)  V\left( \frac{w_n}{\Gamma_1} \right) +  V\left( \frac{w_n}{\Gamma_{2^k}} \right) 
     + T_n(k)
     \bigg) \to 1.
 \end{align*}
At the moment, we claim that, for all sufficiently small $\rho>0$, and suitably large $c_3>0$ and $c_4>0$, 
  \begin{equation}
     \label{eq:upper-bound-of-M-mid-nk;3}
   \lim_{k_0\to\infty} \limsup_{n\to\infty}  \sum_{k=k_0}^{k_n-1 }   \PP \big(  T_n(k) > \eta_n + c_4 k h (\eta_n)
     \big) = 0.
 \end{equation}
From \eqref{eq:upper-bound-of-M-mid-nk;3},  we will obtain \eqref{eq:middle-part-do-not-exceed-upper-bound;equivalent-2} as follows. 

\vspace{1em}

Apply  the Chernoff bound to get that 
\begin{align*}
   &\lim_{k_0 \to \infty} \limsup_{n\to \infty} \PP 
     \left( \forall \,  k_0 \leq k \leq k_n-1,
      V\left( \frac{w_n}{\Gamma_{2^k}} \right) 
     < V\left( \frac{w_n}{2^{k-1}} \right) 
     \right) \\
     = & \lim_{k_0\to \infty} \sum_{k=k_0}^{\infty} \PP 
     \left( \Gamma_{2^k} 
     < {2^{k-1}} 
     \right) =0.
\end{align*}
And we note that  
\begin{equation*}
 \lim_{c_0 \to 0^+}     \PP 
     \left( 
      V\left( \frac{w_n}{\Gamma_1} \right) > V\left( \frac{w_n}{c_0} \right)
    \right) = \lim_{c_0 \to 0^+} \PP(\Gamma_1 < c_0) = 0.
\end{equation*}
Therefore, it suffices to show that, as $n \to 0$, 
\begin{equation}
    \label{eq:the-core-comparison}
      m V\left( \frac{w_n}{K} \right) - 
    \left[ (m-1)  V\left( \frac{w_n}{c_0} \right) +  V\left( \frac{w_n}{2^{k-1}} \right) + \eta_n + c_4k h(\eta_n)  \right]
    >0
\end{equation}
for all choices of  $c_0, c_3, c_4$ and $k=k_0,\ldots,k_n-1$. This will be done in two steps. First, we note that $2^{k_n}$ fits into  \eqref{eq:pertubation-growth-condition},
we thus apply \eqref{eq:pertubation-result;2.1} to deduce that, for all  $k_0 \leq k \leq k_n-1$,
\begin{equation}
    \label{eq:first-middle-difference-comparison}
    m V\left( \frac{w_n}{K} \right) - (m-1)  V\left( \frac{w_n}{c_0} \right) - V\left( \frac{w_n}{2^{k-1}} \right)  
    \geq  c_5 k h \circ  V(w_n)
\end{equation}
for some $c_5> 0$. Apply again \eqref{eq:pertubation-result;2.1} to derive that, for all  $k_0 \leq k \leq k_n-1$,
\begin{equation}
  \eta_n + c_4k h(\eta_n)      
 \geq  c_6 k  h \circ (\vartheta_n \log n)
 \label{eq:second-middle-difference-comparison}
\end{equation}
for some $c_6>0$. Note that $w_n\in \mathsf{RV}_{1-\beta}$,  $\vartheta_n\in \mathsf{RV}_{\betast}$ and $1-\beta > \betast$. Hence, we apply   \eqref{eq:quantile-ratio-lower-bound} to note that 
$ h \circ  V(w_n)  / h \circ (\vartheta_n \log n) \to \infty.$
Combining \eqref{eq:first-middle-difference-comparison} and 
 \eqref{eq:second-middle-difference-comparison} gives  \eqref{eq:the-core-comparison}, hence implies \eqref{eq:middle-part-do-not-exceed-upper-bound;equivalent-2}.

 \vspace{1em}

 Finally, let us prove  \eqref{eq:upper-bound-of-M-mid-nk;3} to finish the whole proof. Define the event 
 \begin{equation}
     \label{eq:mid-part-m-intersections-cardinality-control}
     E_{n,k} = \left \{  \exists \, J \in \mathcal{J}_n^\mathsf{mid}( k)    
     \text{ such that } \left| \bigcap_{i=1}^m I_{j_i,n}   \right| > c_3 \vartheta_n  \log n
    \right \}
 \end{equation}
 Each $\mathcal{J}_n^\mathsf{mid}( k) $ contains at most $2^{mk}$ many distinct $J$'s. By  Proposition \ref{prop:more-intersection-properties} $(\rom3)$, as long as $c_3>0$ is large enough, there exists $C_3 > m$ and such that
 \begin{align*}
     \sum_{k=k_0}^{k_n-1} \PP \left(
      E_{n,k} 
    \right ) \leq  n^{-C_3 } \sum_{k=k_0}^{k_n-1} 2^{mk} \leq
     2^{m(k_n+1)}  n^{-C_3 }  \lesssim n^{m-C_3} \xlongrightarrow{n\to \infty} 0.
 \end{align*}
As such, we reduce  \eqref{eq:upper-bound-of-M-mid-nk;3} into
\begin{align*}
      & \sum_{k=k_0}^{k_n-1}\PP \left(   T_n(k) >\eta_n + c_4 k h (\eta_n) ,  E_{n,k} ^c
     \right) \\
     \leq &  \sum_{k=k_0}^{k_n-1} 2^{mk} \cdot c_3 \vartheta_n \log n \cdot \PP \left(  \sum_{j > n^{\delta_0} } V \left( \frac{w_n}{\Gamma_j}  \right) \one(t \in I_{j,n}) >   \eta_n + c_4 k h (\eta_n)  \right) \\
     \leq & \sum_{k=k_0}^{k_n-1} 2^{mk} \cdot c_3 \vartheta_n \log n \cdot \PP \left(  X_0 >   \eta_n + c_4 k h (\eta_n)  \right) 
     \lesssim  \sum_{k=k_0}^{k_n-1} 2^{mk} \cdot \frac{\overline{\nu} \left(  \eta_n + c_4 k h (\eta_n)  \right)  }{ \overline{\nu} (\eta_n)}.
\end{align*}
From the Assumption \ref{assumption:main-assumptions-for-the-main-theorem} $(\rom2)$ that $\overline{\nu} = c_\# \overline{H_\#}$, hence for all $n \gg 0$,
$$
\frac{\overline{\nu} \left(  \eta_n + c_4 k h (\eta_n)  \right)  }{ \overline{\nu} (\eta_n)} =  
\exp \left(  - \int_{0 }^{  c_4 k   } \frac{h (\eta_n) du}{h(\eta_n + u  h (\eta_n) )}     \right).
$$
 Next, due to \eqref{eq:main-assumption-V-sharp-and-V}, it follows that 
\begin{align*}
    \frac{k_n h(\eta_n)}{\eta_n} \sim &  \frac{k_n h \circ V_\#(c_3 \vartheta_n \log n)}{V_\#(c_3 \vartheta_n \log n)} \\
    = & \frac{k_n  \cdot c_3 \vartheta_n \log n \cdot V^\prime_\#( c_3 \vartheta_n \log n )}{V_\#(c_3 \vartheta_n \log n)} 
    \asymp  \frac{k_n \zeta\big( \log(c_3 \vartheta_n \log n ) \big) }{ \log (c_3 \vartheta_n \log n)  } \asymp \rho.
\end{align*}
Hence, if $\rho>0$ is small enough, then the regular variation $h \in \mathsf{RV}_1$ implies that, for all $k=k_0,\ldots,k_n-1$,  
$$
 \int_{0 }^{  c_4 k   } \frac{h (\eta_n) du}{h(\eta_n + u  h (\eta_n) )}    
 \geq \int_{0 }^{  c_4 k   } \frac{du}{2} = \frac{c_4k}{2}.
$$
As long as $c_4>0$ is such that $e^{-c_4/2} 2^m <1$, we shall finish the proof by noting that 
$$
\lim_{k_0 \to \infty} \limsup_{n\to \infty} \sum_{k=k_0}^{k_n-1}\PP \left(   T_n(k) >\eta_n + c_4 k h (\eta_n) ,  E_{n,k} ^c
     \right) 
\lesssim \lim_{k_0 \to \infty} \sum_{k=k_0} ^ \infty 2^{mk} e^{-c_4 k / 2} = 0.
$$
\end{proof}

\subsubsection{Bottom Comparison}
\label{subsubsec:bottom-comparison}

Again, regardless of the open interval $B\subset [0,1]$, we define 
\begin{equation}
    \label{eq:upper-bound-bottom-part-extreme}
    M_n^\mathsf{btm} =  \max \left\{ X_t : 0 \leq t  \leq n,  j_m(t) \geq  2^{k_n} \right\}.
\end{equation}

\begin{proposition} 
\label{prop:bottom-part-do-not-exceed-upper-bound}
    For any $\rho>0$, 
    \begin{equation}
       \label{eq:bottom-part-do-not-exceed-upper-bound}
    \lim_{n \to \infty} 
     \PP 
    \bigg(  
   M_n^\mathsf{btm}
     > M^\mathsf{upper}_n (B,c_1=0,c_2=0)
    \bigg) = 0.
    \end{equation}
\end{proposition}

\begin{proof}[Proof of Proposition \ref{prop:bottom-part-do-not-exceed-upper-bound}]
    \phantom{blank}

   As it is argued in \eqref{eq:middle-part-do-not-exceed-upper-bound;equivalent-1}, 
   it suffices to  show that, for any $C > 0$, 
  \begin{equation}
       \label{eq:bottom-part-do-not-exceed-upper-bound;1}
    \lim_{n \to \infty} 
     \PP 
    \bigg(  
   M_n^\mathsf{btm}
     > m V\left( \frac{w_n}{C} \right)
    \bigg) = 0.
    \end{equation}
  Due to   \eqref{eq:finite-intersection-properties-empty}, it follows that
  \begin{align*}
  \lim_{n\to \infty}
      \PP \bigg(
        M_n^\mathsf{btm} \leq & (m-1) V \left( \frac{w_n}{\Gamma_1} \right) + 
         V \left( \frac{w_n}{\Gamma_{2^{k_n}}} \right) \\
         & + \max_{0\leq t \leq n} \sum_{j \geq n^{\delta_0}} 
          V \left( \frac{w_n}{\Gamma_j} \right) \one(t\in I_{j;n})
      \bigg) = 1.
  \end{align*}
From \eqref{eq:main-assumption-V-sharp-and-V}, we obtain
$$
     \frac{V ( w_n  / 2^{k_n} ) }{  V (w_n)  }
     \sim \exp \left(  - \int_{ w_n / 2^{k_n} } ^{w_n} \frac{\zeta(\log u)}{u \log u} du  \right) 
     = \exp \left(  - \int_{ \log w_n - k_n \log 2} ^{ \log w_n} \frac{\zeta(u)}{u } du  \right).
$$
By Lemma \ref{lemma:properties-exp-slowly-varying-distr} $(\rom2)$ and the assumption $\zeta\in \mathsf{RV}_\alpha$, it follows that
\begin{align*}
    \int_{ \log w_n - k_n \log 2} ^{ \log w_n} \frac{\zeta(u)}{u } 
\sim \zeta (\log w_n) \log \frac{\log w_n}{\log w_n - k_n \log 2} 
\sim \rho (1-\beta)^{\alpha-1} > 0
\end{align*}
 Therefore, for any $0< \epsilon < \rho (1-\beta)^{\alpha-1}$,  the Law of the Large Numbers gives 
\begin{equation}
    \label{eq:small-gap-key-gap}
    \lim_{n\to \infty} 
      \PP \left(
      \frac{V ( w_n  /\Gamma_{ 2^{k_n} } ) }{  V (w_n)  } < 1-\epsilon  
      \right) = 1.
\end{equation}
Due to  \eqref{eq:small-gap-key-gap} and \eqref{eq:Pi-variation-and-weak-convergence},
we will establish  \eqref{eq:bottom-part-do-not-exceed-upper-bound;1} once we prove
 \begin{equation}
    \label{eq:bottom-part-do-not-exceed-upper-bound;2}
    \lim_{n \to \infty} 
     \PP 
    \left(  
   \max_{0\leq t \leq n} \sum_{j \geq n^{\delta_0}} 
          V \left( \frac{w_n}{\Gamma_j} \right) \one(t\in I_{j;n}) >\epsilon V(w_n)
    \right) = 0, \quad \forall \, \epsilon > 0.
    \end{equation}
The proof \eqref{eq:bottom-part-do-not-exceed-upper-bound;2} is systematic on its own, see the proposition below.
\end{proof}

\vspace{1em}

\begin{proposition} 
\label{prop:a-fast-enough-tail-estimation}
\phantom{blank}

    For any $0<r<1-\beta$ and any $\epsilon > 0$, 
    \begin{equation}
       \label{eq:maximum-of-remainder-series-union-bound}
     \lim_{n\to \infty}  n \cdot \PP\left(  \sum_{j \geq n^r} V \left( \frac{w_n}{\Gamma_j} \right) \mathbf{1} \left(0 \in I_{j,n} \right) 
       \geq \epsilon V(w_n)
       \right) = 0
    \end{equation}
\end{proposition}

\begin{proof}[Proof of Proposition \ref{prop:a-fast-enough-tail-estimation}]
\phantom{blank}

    Because of the Chernoff  bound and that $r$ is arbitrary, it suffices to show that 
    \begin{equation}
       \label{eq:maximum-of-remainder-series-union-bound-equivalent}
     \lim_{n\to \infty}  n \cdot \PP\left(  \sum_{\Gamma_j \geq n^r} V \left( \frac{w_n}{\Gamma_j} \right) \mathbf{1} \left(0 \in I_{j,n} \right) 
       \geq \epsilon V(w_n)
       \right) = 0.
    \end{equation}
  For simplicity, we introduce 
 for each $n \in \bbn$, 
    $$
    u_n = \epsilon V(w_n), \quad  x_n = V \left( \frac{w_n}{n^r} \right), \quad m_n = \left \lfloor  \frac{u_n}{x_n} \right \rfloor. 
    $$
    And, for each $s \in \bbn$, define $G_s: (1,\infty)^s \to \bbr$ by 
    $$
   G_s (z_1,\ldots, z_s) = \int_{1}^{z_1} \frac{du}{h(u)} + \cdots + 
    \int_{1}^{z_s} \frac{du}{h(u)} \triangleq g(z_1) + \cdots + g(z_s).
    $$
    For all large enough $n$, let $N_n$ be a Poisson random variable with mean $1-n^rw_n^{-1} > 0$. And, let $(\xi_{i,n} : i \in \bbn)$ be a family of i.i.d.\ random variables independent of $N_n$. Let  $\xi_{1;n}$ have the distribution $H_n$,
    $$
    H_n (a,b) \triangleq  \frac{\nu(a,b)} {\nu (1, x_n)} =   \frac{H_\#(a,b)} { H_\# (1, x_n)}, \quad \forall \, 1 \leq a < b \leq x_n.
    $$
    It follows that 
    \begin{equation}
        \sum_{\Gamma_j > n^r}  V \left( \frac{w_n}{\Gamma_j} \right) \mathbf{1} \left(0 \in I_{j,n} \right)  \eid 
        \sum_{i=1}^{N_n} \xi_{i,n}. \notag
    \end{equation}
    and hence
    \begin{align}
         \PP\left(  \sum_{\Gamma_j \geq n^r} V \left( \frac{w_n}{\Gamma_j} \right) \mathbf{1} \left(0 \in I_{j,n} \right)  \right)  
        &=  \sum_{d=0}^\infty \PP( N_n = m_n + d ) \cdot \PP\left( \sum_{i=1}^{m_n+d} \xi_{i,n} > u_n \right)  \notag \\
        & \triangleq 
        \sum_{d=0}^\infty B_d \cdot L_d.
        \label{eq:truncated-remainder-series-compound-Poisson-expansion;}
    \end{align}
    We obviously have
    \begin{equation}
        \label{eq:truncated-remainder-series-compound-Poisson-expansion;Bd}
        B_d \leq \frac{( 1 - n^r w_n^{-1} )^{m_n+d} }  {(m_n+d)!} 
        \leq \frac {1} {(m_n + d)!} ,
    \end{equation}
    so the key is estimate $L_d$.

    \vspace{1em}

    To this end, choose any positive $c >1 / \nu(1,\infty)$, then for all $n\gg 0$, 
    \begin{align}
        L_d = & \int_{ (1, x_n)^{m_n+d} } \mathbf{1} 
        \left(
          \sum_{i=1}^{m_n+d} z_i > u_n
        \right)
        \prod_{i=1}^{m_n+d} H_n (dz_i) \notag \\
        \leq & c^{m_n+d}  \int_{ (1, x_n)^{m_n+d} } \mathbf{1} 
        \left(
          \sum_{i=1}^{m_n+d} z_i > u_n
        \right)
        \prod_{i=1}^{m_n+d}  H_\# (dz_i) \notag \\
        = & c ^{m_n+d}  \int_{ (1, x_n)^{m_n+d} } \mathbf{1} 
        \left(
          \sum_{i=1}^{m_n+d} z_i > u_n
        \right)
        \prod_{i=1}^{m_n+d} \exp (- g(z_i)) g^\prime (z_i) d z_i  \notag \\
        \leq & \big(c  g(x_n) \big)^{m_n+d} \cdot \exp (  - C_{d,n} )  \label{eq:truncated-remainder-series-compound-Poisson-expansion-concavity-optimization;1} 
    \end{align}
    where 
    \begin{equation}
         \label{eq:truncated-remainder-series-compound-Poisson-expansion-concavity-optimization;2}
         C_{d,n} = \inf \left \{  
     G_{m_n+d} (z_1,\ldots, z_{m_n+d}) :
     \sum_{i=1}^{m_n+d} z_i \geq u_n , 1 < z_i < x_n    
     \right\}.
    \end{equation}
Since the function $G_{m_n+d}$ is concave in all of its variables, the infimum  \eqref{eq:truncated-remainder-series-compound-Poisson-expansion-concavity-optimization;2} must be obtained at a boundary point subject to the conditions below: 
\begin{itemize}
    \item[$(\rom1)$] $k_d$ many coordinates equal to $x_n$;
    \item[$(\rom2)$] $m_n+d - k_d - 1$ many coordinates equal to $1$;
    \item[$(\rom3)$] a final coordinate lies in $[1, x_n]$ and  equals to $u_n - k_d x_n - (m_n+d-k_d-1)$;
    \item[$(\rom4)$]  $k_d$ is the smallest integer among all possible choices.
\end{itemize}

\vspace{1em}
 To finish the proof, let us consider $d\geq 0$ in two scenarios. First, if 
\begin{equation}
    \label{eq:estimate-Qd-small-range-d}
    d \leq \left \lfloor \frac{u_n}{2} - x_n \right \rfloor - m_n \triangleq d_n^\ast, 
\end{equation}
then we must obtain 
\begin{equation}
    k_d  \geq  \frac{m_n}{3}. 
    \notag
\end{equation}
Otherwise, for $n\gg 0$, we would have 
\begin{align*}
  &  k_d x_n + (m_n+ d - k_d - 1 ) \\
  \leq &  k_d x_n + (m_n+d)  \leq \frac{m_n x_n}{3} + \left \lfloor \frac{u_n}{2} - x_n \right \rfloor \leq \frac{5 u_n}{6} - x_n \ll u_n - 2x_n
\end{align*}
which were a contradiction to the condition $(\rom3)$ above. It thus follows that  
$C_{d,n} \geq g(x_n) m_n /3$ and 
\begin{align*}
  \sum_{d=0}^{d_n^\ast} B_d \cdot L_d 
   \leq & \exp \left(  -\frac{g(x_n) m_n}{3}  \right) 
    \cdot \sum_{d=0} ^{d^\ast} \frac{ [c g(x_n) ]^{m_n+d}}{(m_n+d)!} 
    \leq \exp \left( -g (x_n) \left( \frac{m_n}{3} - c \right) \right).   
\end{align*}
Note that
$
g(x_n) = -\log \overline{H_\#} \big( V(w_n/n^r) \big) \sim (1-\beta-r)\log n
$
and
$
m_n \big/ \log \log n \to \infty.
$
 by \eqref{eq:quantile-ratio-lower-bound}.
Therefore, 
 \begin{equation}
     \label{eq:sum-BdLd-before-critical-d}
   n \cdot \sum_{d=0}^{d_n^\ast} B_d \cdot L_d \to 0.
 \end{equation}
It remains to show that 
 \begin{equation}
     \label{eq:sum-BdLd-after-critical-d}
   n \cdot \sum_{d=d_n^\ast+1}^\infty B_d \cdot L_d \to 0.
 \end{equation}
 Clearly, we have
 $$
     n \cdot \sum_{d=d^\ast_n+1}^{\infty} B_d \cdot L_d \leq n \sum_{d=d^\ast_n+1}^{\infty} B_d
\lesssim \frac{n}{(m_n+d^\ast_n)!}.
 $$
From \eqref{eq:estimate-Qd-small-range-d} we get $m_n + d_n \asymp u_n$ and from Proposition \eqref{prop:properties-exp-slowly-varying-distr} $(\rom1)$ we get 
$
u_n / (\log n)^2 \to \infty.
$
Hence, for all $n \gg 0$, we claim from the Stirling's formula that 
$$
 \frac{n}{(m_n+d^\ast_n)!} \leq \frac{n}{ \lfloor (\log n)^2 \rfloor ! }\to 0.
$$
Putting together \eqref{eq:sum-BdLd-before-critical-d} and \eqref{eq:sum-BdLd-before-critical-d} finishes the proof. 
\end{proof}

\subsection{Proofs of Main Theorems}
\label{subsec:proofs-of-main-theorems}

\begin{proof}[Proof of Theorem \ref{main-theorem:random-sup-measure-convergence}]

Let $B_1,\ldots, B_s$ be arbitrarily $s\in \bbn$ disjoint open intervals. It suffices to show that 
\begin{equation}
    \label{eq:rsm-convergence-main-theorem-1}
    \left( \frac{\mathcal{M}_n(B_r) - b_n}{a_n} \right)_{r=1,\ldots,s}  
    \Rightarrow 
     \left( \mathcal{M}_\beta(B_s) \right)_{r=1,\ldots,s}  .
\end{equation}
For each $B_r$, by \eqref{eq:the-lower-bound-rsm-is-trivial} and \eqref{eq:upper-bound-with-high-P-no-exceed}, we note that 
$$
\lim_{c_1, c_2 \to \infty} \limsup_{n\to \infty} \PP\left( M_n^\mathsf{lower}(B_r) \leq  \mathcal{M}_n(B_r) \leq M_n^\mathsf{upper}(B_r, c_1, c_2) \right) = 1.
$$
Because of \eqref{eq:weak-convergence-of-lower-bound}, we will obtain \eqref{eq:rsm-convergence-main-theorem-1} once we show that 
\begin{align}
    & \frac{M_n^\mathsf{upper}(B_r, c_1, c_2) - M_n^\mathsf{lower}(B_r)}{a_n}  \notag \\
    = & \frac{ V(w_n / \Gamma_{j^{\doublestar}_n (B) } ) - V(c_1 \vartheta_n) - c_2 h\circ V(c_1\vartheta_n) } { h \circ V(w_n) } \Rightarrow 0.  \label{eq:rsm-convergence-main-theorem-2}
\end{align}
By \eqref{eq:weak-convergence-of-j-star-star-m-plus-one} and \eqref{eq:quantile-ratio-lower-bound} respectively, we have 
\begin{equation}
      \frac{V(w_n / \Gamma_{j^{\doublestar}_n (B) } )  - V(c_1 \vartheta_n) }{ h \circ V(\vartheta_n)}  \Rightarrow -\log (c_1 U_B)   \quad \text{and} \quad
\frac{c_2 h\circ V(c_1\vartheta_n)}{ h \circ V(w_n)} \to 0.
\label{eq:rsm-convergence-main-theorem-3}
\end{equation}
Hence, \eqref{eq:rsm-convergence-main-theorem-2} is proved.

\end{proof}

\vspace{1em}

Our proof of Theorem \ref{main-theorem:extremal-processes-convergence} will again require ``Billingsley's convergence together lemma". For each $n\in \bbn$ and $k\in \bbn$, we first define the truncated random sup-measure $\mathcal{M}_{n,k}$ and the induced extremal process $\mathcal{E}_{n,k}$,
\begin{align*}
  & \mathcal{M}_{n,k} (B)= \max_{t\in nB} \sum_{j=1}^\infty V\left( \frac{w_n}{\Gamma_j}  \right) \one( t\in I_{j,n} ), \quad B \in \mathcal{B}_{[0,1]}, \quad t \in \bigcup_{j=1}^k I_{j,n}; \\
  & \mathcal{E}_{n,k} (t)  = \mathcal{M}_{n,k} \big( (0,t] \big), \quad t \in (0,1].
\end{align*}
On the limit profile, we shall also need the truncated sup-derivative and the induced objects,
\begin{align*}
 & \eta_k (t) = 
 \begin{cases}
   \sum_{j=1}^k - \log \Gamma_j \one (t\in R_j), \quad & \text{if } \sum_{j=1}^k \one(t\in R_j) = m \\   
   -\infty , & \text{otherwise}
 \end{cases}, \\
 & \mathcal{M}_{\beta,k} (B) = \sup_{t\in B} \eta_k (t), \quad B \in \mathcal{B}_{[0,1]}; \\
 & \mathcal{E}_{\beta,k} (t) = \mathcal{M}_{\beta,k} \big( (0,t] \big) = \sup_{s\in (0,t]} \eta_k (s), \quad t \in (0,1].
\end{align*}
All notations above are consistent with \eqref{eq:series-representation-abuse}, \eqref{eq:rsm-restricted-sup-derivative-eta} and \eqref{eq:limit-rsm-restriction-unit-interval}. We shall take $\delta_0 \in (0,1-\beta-1/(m+1))$ as used in 
\eqref{eq:m-intersected-sets-collection-prelimit}.

\vspace{1em}

\begin{proof}[Proof of Theorem \ref{main-theorem:extremal-processes-convergence}]
\phantom{blank}

It suffices to prove that, for all $0< T_1 < T_2 \leq 1$, 
\begin{equation}
    \label{eq:main-thm-extreme-pro-convergence;1}
   \frac{ \mathcal{E}_n (\cdot) - b_n } { a_n } 
   \xLongrightarrow{( D[T_1, T_2], J_1)} \mathcal{E}_\beta (\cdot)
\end{equation}
To apply \cite{billingsley:1999:book} Theorem 3.2, we  shall show the followings.
\begin{align}
  & \lim_{k\to \infty} \PP\left( \big( \mathcal{E}_{\beta,k} (t) \big)_{t\in [T_1, T_2]} =  \big( \mathcal{E}_\beta (t) \big)_{t\in [T_1, T_2]}   \right) = 1
  \label{eq:main-thm-extreme-pro-convergence;2}, \\
 & \lim_{k\to \infty} \limsup_{n\to \infty} \PP\left( \big( \mathcal{E}_n (t) \big)_{t\in [T_1, T_2]} \neq  \big( \mathcal{E}_{n,k} (t) \big)_{t\in [T_1, T_2]} \right) = 0 , 
 \label{eq:main-thm-extreme-pro-convergence;3} \\
  & \frac{ \mathcal{E}_{n,k} (\cdot) - b_n } { a_n } 
   \xLongrightarrow{( D[T_1, T_2], J_1)} \mathcal{E}_{\beta,k} (\cdot).
  \label{eq:main-thm-extreme-pro-convergence;4}
\end{align}

\vspace{1em}

We easily note that 
$$
\lim_{k\to \infty} \big( \mathcal{E}_{\beta,k} (t) \big)_{t\in [T_1, T_2]} 
\xlongequal{a.s.}
\big( \mathcal{E}_\beta (t) \big)_{t\in [T_1, T_2]} ,
$$
which gives \eqref{eq:main-thm-extreme-pro-convergence;2}. 
Let us prove  \eqref{eq:main-thm-extreme-pro-convergence;3} by a counterargument. 
Suppose that \eqref{eq:main-thm-extreme-pro-convergence;3}  were false, then it would follow that
\begin{align}
     &\liminf_{k\to \infty} \,
    \liminf_{n\to \infty} 
     \PP \bigg[  \exists \, t \in [T_1, T_2] \text{ such that }
     \mathcal{M}_n \big( (0,t] \big) \neq \mathcal{M}_{n,k}\big( (0,t] \big) \bigg] > 0.  \label{eq:main-thm-extreme-pro-convergence;3.2}
\end{align} 
On the event $\{   \mathcal{M}_n \big( (0,t] \big) \neq \mathcal{M}_{n,k}\big( (0,t] \big)  \}$, clearly 
$$
     \mathcal{M}_n \big( (0,t] \big) \leq \max \left\{ M^\mathsf{mid}_n( \lfloor \log_2 k \rfloor ) , M^\mathsf{btm}_n  \right\}.
$$
On the other hand, we have for all $T_1\leq t \leq T_2$,
$$
\mathcal{M}_n \big( (0,t] \big) \geq \mathcal{M}_n \big( (0,T_1] \big). 
$$
Hence, \eqref{eq:main-thm-extreme-pro-convergence;3.2} would imply that 
$$
\liminf_{k\to \infty} \,
    \liminf_{n\to \infty} 
     \PP \bigg[  \max \left\{ M^\mathsf{mid}_n( \lfloor \log_2 k \rfloor ) , M^\mathsf{btm}_n  \right\} > \mathcal{M}_{n}\big( (0,T_1] \big) \bigg] > 0,
$$
which contradicts with Proposition \ref{prop:middle-part-upper-bound} and Proposition \ref{prop:bottom-part-do-not-exceed-upper-bound}.

\vspace{1em}

Finally, let us prove \eqref{eq:main-thm-extreme-pro-convergence;4}.  For each fixed $k \geq m$, consider events 
\begin{align*}
    & E_{n,k} \triangleq \left\{ \exists\, 1\leq j_1 < \cdots < j_m \leq k \text{ such that }  \bigcap_{i=1}^m I_{j_i,n}  \neq \emptyset \right\} ,\\
    &  E_{\beta,k} \triangleq \left\{ \exists\, 1\leq j_1 < \cdots < j_m \leq k \text{ such that } \bigcap_{i=1}^m R_{j_i}  \neq \emptyset \right\}. 
\end{align*}
and we can decompose 
$$
\mathcal{E}_{n,k} (\cdot) = \mathcal{E}_{n,k} (\cdot)  \one_{E_{n,k}} + \mathcal{E}_{n,k} (\cdot)  \one_{E_{n,k}^c}.
$$
It is obvious that 
\begin{equation}
    \frac{ \mathcal{E}_{n,k} (\cdot) - b_n } { a_n }  \one_{E_{n,k}^c} 
   \xLongrightarrow{( D[T_1, T_2], J_1)} -\infty \cdot \one_{E_{\beta,k}^c}, \quad -\infty \cdot 0 \triangleq 0;
  \label{eq:main-thm-extreme-pro-convergence;4.1}
\end{equation}
with no jumps in the limit sample paths.  We next note that  
\begin{align*}
  & \mathcal{M}_{n,k} (\cdot) \one_{E_{n,k}} = \max_{J=(j_1,\ldots,j_m)_< \subset \{1,\ldots, k\} } \mathcal{M}_{n,k,J} (\cdot)\one_{E_{n,k}} , \quad -\infty \cdot 0 \triangleq 0; \\
  &  \mathcal{M}_{n,k,J} (B) = \max_{t\in nB} \sum_{j=1}^\infty V\left( \frac{w_n}{\Gamma_j}  \right) \one( t\in I_{j,n} ), \quad B \in \mathcal{B}_{[0,1]}, \quad t \in \bigcap_{i=1}^m I_{j_i,n}.
\end{align*}
And   we can write $ \mathcal{E}_{n,k} (\cdot) \one_{E_{n,k}} = \max_{J=(j_1,\ldots,j_m)_< \subset \{1,\ldots, k\} } \mathcal{E}_{n,k,J} (\cdot) $ where  
\begin{align*}
  &  \mathcal{E}_{n,k,J} (t) \triangleq  \mathcal{M}_{n,k,J} \big( (0,t] \big), \quad t\in [T_1 , T_2].
\end{align*}
For each $J$ above, we claim that
\begin{equation}
    \label{eq:main-thm-extreme-pro-convergence;4.2}
    \frac{  \mathcal{E}_{n,k,J} (\cdot) - b_n} {a_n}  \xLongrightarrow{\big( D[T_1, T_2], J_1 \big)}
  \mathcal{E}_{\beta,k,J} (\cdot)
\end{equation}
 with
$$
  \mathcal{E}_{\beta,k,J} (t) \triangleq 
  \begin{cases}
      \sum_{i=1}^m -\log \Gamma_{j_i}, \quad  & t \geq \inf \bigcap_{i=1}^m \mathcal{R}_{j_i} \\
       - \infty,  & t < \inf \bigcap_{i=1}^m \mathcal{R}_{j_i}  
  \end{cases}, \quad t\in [T_1, T_2].
$$
The $J_1$-convergence  in  \eqref{eq:main-thm-extreme-pro-convergence;4.2} hold because there is at most one jump in $[T_1, T_2]$ and the joint weak convergence of jump location and height holds. Due to 
\begin{equation}
    \label{eq:main-thm-extreme-pro-convergence;4.3}
    \mathcal{E}_{\beta,k} (\cdot) = \max_{J=(j_1,\ldots,j_m)_< \subset \{1,\ldots, k\} } \mathcal{E}_{\beta,k,J} (\cdot),
\end{equation}
so there are at most $k^m$ jumping points in the sample path of $ \mathcal{E}_{\beta,k}$ on $[T_1, T_2]$.  We obtain \eqref{eq:main-thm-extreme-pro-convergence;4} from  \eqref{eq:main-thm-extreme-pro-convergence;4.1}, \eqref{eq:main-thm-extreme-pro-convergence;4.2} and \eqref{eq:main-thm-extreme-pro-convergence;4.3}.
\end{proof}

\bibliographystyle{alpha}
\bibliography{chernzl}

\end{document}